\documentclass[a4paper,11pt]{article}

\usepackage[T1]{fontenc}
\usepackage{lmodern}

\usepackage{dsfont}

\usepackage[latin1]{inputenc}
\usepackage{amsmath}
\usepackage{amsthm}
\usepackage{amssymb}
\usepackage{mathrsfs}
\usepackage{graphicx}
\usepackage[all]{xy}
\usepackage{hyperref}

\usepackage{makeidx}

\usepackage{stmaryrd}
\usepackage{caption}

\usepackage{abstract} 

\newtheorem{thm}{Theorem}[section]
\newtheorem{cor}[thm]{Corollary}
\newtheorem{claim}[thm]{Claim}
\newtheorem{fact}[thm]{Fact}

\newtheorem{lemma}[thm]{Lemma}
\newtheorem{prop}[thm]{Proposition}

\theoremstyle{definition}
\newtheorem{definition}[thm]{Definition}

\newtheorem{remark}[thm]{Remark}

\def\rquotient#1#2{%
	\makeatletter
	\raise.3ex\hbox{$#1$}/\lower.3ex\hbox{$#2$}%
	\makeatother
}	

\makeatletter
\newcommand{\subjclass}[2][2010]{%
	\let\@oldtitle\@title%
	\gdef\@title{\@oldtitle\footnotetext{#1 \emph{Mathematics subject classification.} #2}}%
}
\newcommand{\keywords}[1]{%
	\let\@@oldtitle\@title%
	\gdef\@title{\@@oldtitle\footnotetext{\emph{Key words and phrases.} #1.}}%
}
\makeatother

\newcommand{\Address}{{
		\bigskip
		\small
		
\noindent \textsc{Universit\'e de Montpellier\\ 
Institut Math\'ematiques Alexander Grothendieck\\
Place Eug\`ene Bataillon\\
34090 Montpellier (France)}\par\nopagebreak
\noindent \textit{E-mail address}: \texttt{anthony.genevois@umontpellier.fr}

\medskip \noindent \textsc{\'Ecole Normale Sup\'erieure de Lyon\\
15 parvis Ren\'e Descartes\\
BP 7000\\
69342 Lyon Cedex 07 (France)}\par\nopagebreak
\noindent \textit{E-mail address}: \texttt{geoffrey.tournier@ens-lyon.fr}
		
}}

\makeindex

\title{Hyperbolic structures on Houghton groups}
\date{\today}
\author{Anthony Genevois and Geoffrey Tournier\footnote{Most of this work is part of the second author's master thesis, realised between April and June of 2024 at the University of Montpellier, under the supervision of the first author.}}


\begin{document}

\maketitle

\begin{abstract}
Given a group $G$, its poset of hyperbolic structures $\mathcal{H}(G)$ encodes all the possible cobounded actions of $G$ on hyperbolic spaces. In this article, we describe the poset $\mathcal{H}(H_n)$ for every Houghton group $H_n$, $n \geq 2$. In particular, we show that $H_n$ admits exactly $n$ focal hyperbolic structures. As an application, we construct the first example of a group admitting exactly one focal hyperbolic structure, answering a question of Abbott, Balasubramanya, and Osin.
\end{abstract}

\tableofcontents

\section{Introduction}

\noindent
Introduced in \cite{MR3995018}, the \emph{poset of hyperbolic structures} $\mathcal{H}(G)$ of a group $G$ encodes all the possible cobounded actions of $G$ on hyperbolic spaces. More formally, $\mathcal{H}(G)$ is the set of all generating sets $X$, not necessarily finite, for which $\mathrm{Cayl}(G,X)$ is hyperbolic modulo the equivalence relation: $X \sim Y$ whenever $X \preceq Y$ and $Y \preceq X$. Here, $X \preceq Y$ means that 
$$\sup\limits_{y \in Y} \|y\|_X < \infty$$
where $\|\cdot\|_X$ denotes the word-length relative to $X$.  

\medskip \noindent
Except when it is trivial (see \cite{NL} and references therein), the poset of hyperbolic structures is precisely known for only a few groups. Examples include some wreath products \cite{MR4069979}, some solvable groups such as (generalised) Baumslag-Solitar \cite{MR4602410, MR4585997, MR4751864}, and some higher rank lattices \cite{BCFS}. The recent work \cite{BFFZ} is also worth mentioning, investigating the hyperbolic structures of Thompson's groups $F$. For more details, we refer the reader to the survey \cite{SurveyS}.

\medskip \noindent
In this article, we focus our attention on the so-called \emph{Houghton groups}. Introduced in \cite{MR521478}, Houghton groups provide a family of examples that is often used in order to illustrate interesting behaviours. Similarly to lamplighter groups, their definition is simple, making them tractable for direct investigation; but they are also quite different from the large families of groups usually considered in (geometric) group theory. Most notably, the $n$th Houghton group $H_n$ is an example of a group of type $F_{n-1}$ but not of type $FP_n$ \cite{MR885095}.

\medskip \noindent
Given an integer $n \geq 1$, let $\mathcal{R}_n$ denote the disjoint union of $n$ rays $R_i:= \{i\} \times \mathbb{N}$, $1 \leq i \leq n$. The \emph{$n$th Houghton group} $H_n$ is the group of the quasi-automorphisms $\mathcal{R}_n \to \mathcal{R}_n$ that fix each end of $\mathcal{R}_n$. Recall that a quasi-automorphism of a graph is a bijection of its vertices that preserves adjacency and non-adjacency for all but finitely many pairs of vertices. In other words, $H_n$ is the group of the bijections $\mathcal{R}_n \to \mathcal{R}_n$ that eventually acts on each ray $R_1, \ldots, R_n$ like a translation. 

\medskip \noindent
\begin{minipage}{0.45\linewidth}
\includegraphics[trim=0 0 28cm 0, clip, width=0.95\linewidth]{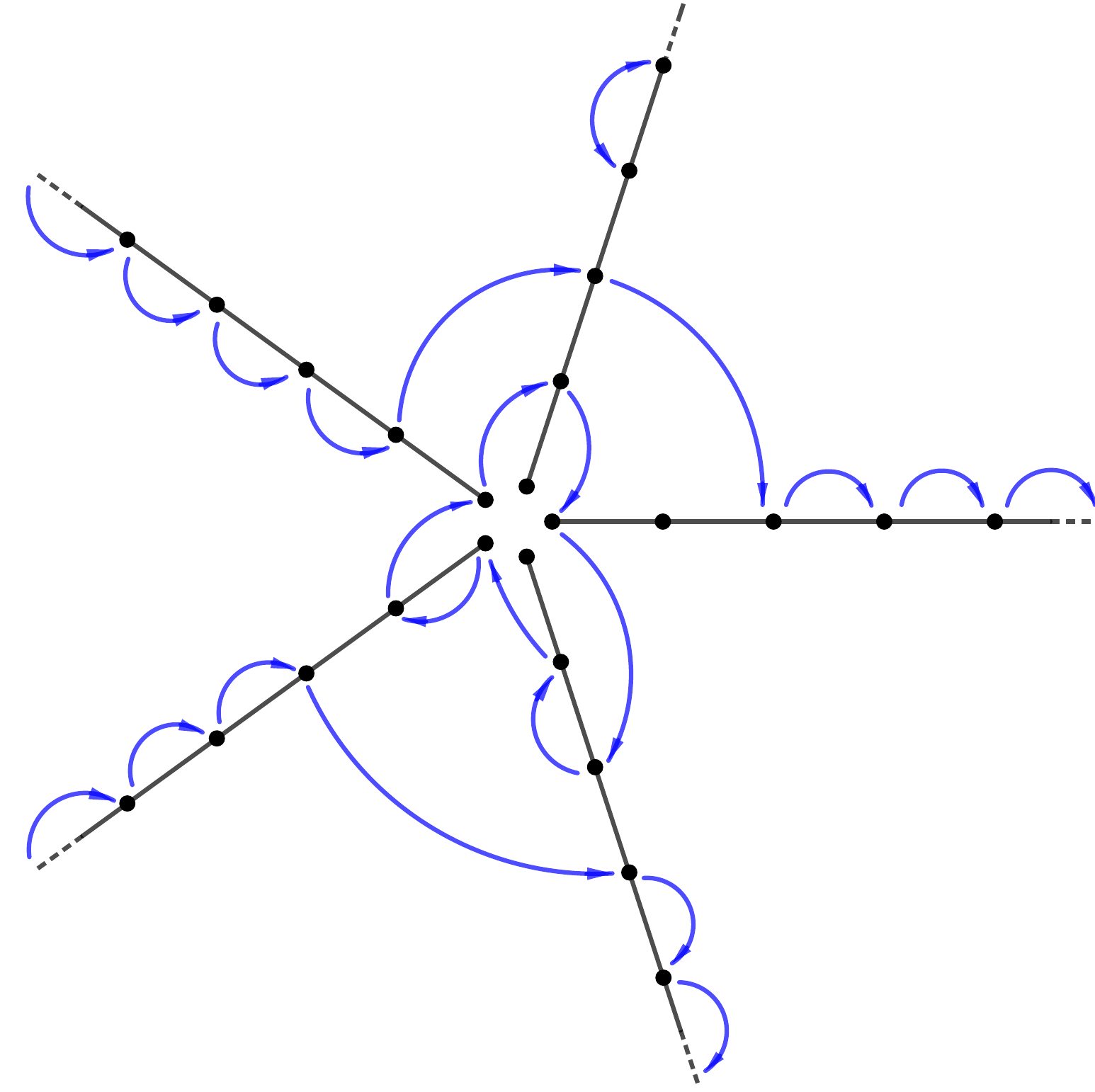}
\end{minipage}
\begin{minipage}{0.54\linewidth}
It is worth noticing that $H_n$ is (locally finite)-by-$\mathbb{Z}^{n-1}$. More precisely, for every $1 \leq i \leq n$, there is a natural morphism $\lambda_i : H_n \twoheadrightarrow \mathbb{Z}$ that records the eventual algebraic translation length of an element along the ray $R_i$. Then $\lambda := \lambda_1 \oplus \cdots \oplus \lambda_n$ maps $H_n$ to $\mathbb{Z}^n$ with kernel the group $\mathfrak{S}_\infty$ of finitely supported permutations of $\mathcal{R}_n$. This morphism $\lambda$, however, is not surjective. Due to the fact that the elements of $H_n$ define bijections $\mathcal{R}_n \to \mathcal{R}_n$, we must have $\lambda_1 + \cdots + \lambda_n=0$. Thus, the image of $\lambda$ has rank $\leq n-1$. In order to justify that the rank is exactly $n-1$, it suffices to consider the images of the elements $t_{i,j}$ inducing a translation on $R_i \cup R_j$ and fixing pointwise the other rays.
\end{minipage}

\medskip \noindent
The main result of our article describes precisely the posets of hyperbolic structures of Houghton groups.

\begin{thm}\label{thm:Intro}
Let $n\geq 2$ be an integer. The Houghton group $H_n$ has exactly $n$ focal hyperbolic structures, namely
$$\mathcal{F}_i:= [\mathrm{Fix}(R_i) \cup T] \text{ for } 1 \leq i \leq n, \text{ where } T:= \{ t_{r,s}, 1 \leq r,s \leq n\};$$
which are pairwise non-comparable. Each $\mathcal{F}_i$ dominates a single lineal hyperbolic structure, namely $[ \mathrm{ker}(\lambda_i) \cup T].$ If $n=2$, $H_n$ has a single lineal hyperbolic structure; if $n \geq 3$, $H_n$ has uncountably many lineal hyperbolic structure, which are pairwise non-comparable.
\end{thm}

\noindent
See Figure~\ref{poset} for a representation of our poset of hyperbolic structures. It is worth mentioning that each focal hyperbolic structure can be realised by an action on a tree, corresponding to a decomposition of the Houghton group as an ascending HNN extension. Theorem~\ref{thm:Intro} answers a question asked in \cite[Example~11.8]{BFFZ}.

\begin{figure}[h!]
\begin{center}
\includegraphics[width=\linewidth]{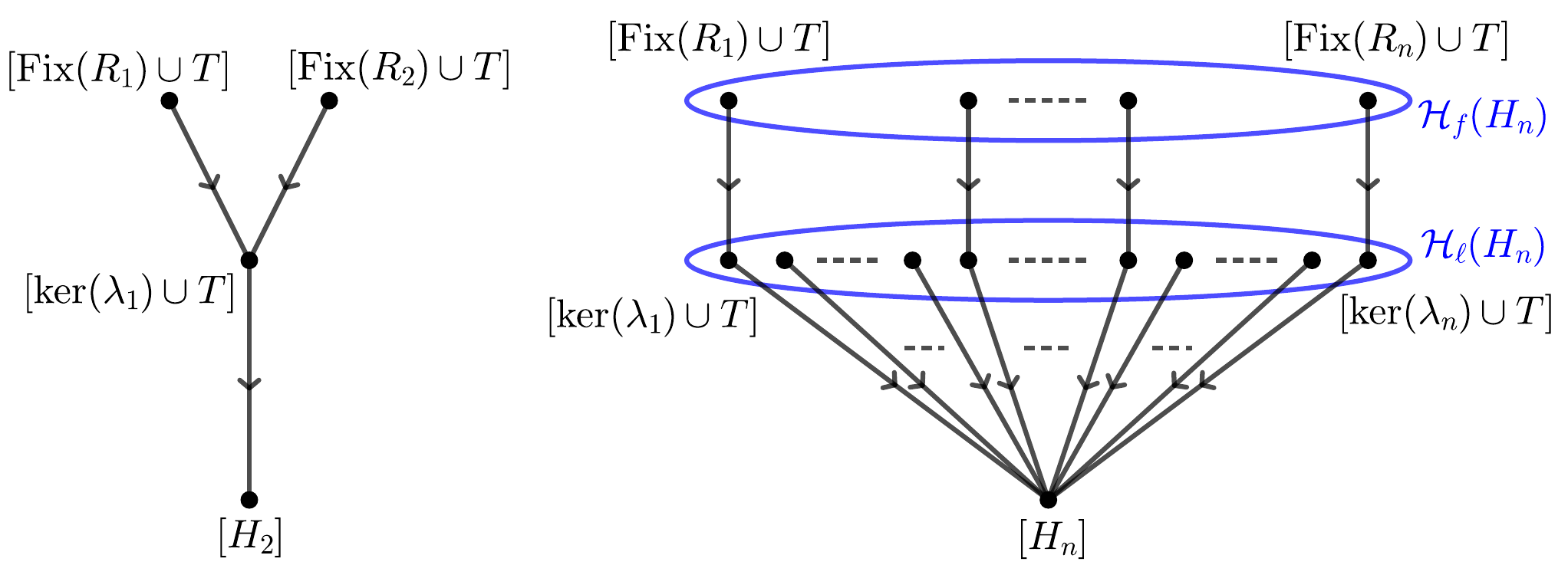}
\caption{Structures of the posets $\mathcal{H}(H_2)$ and $\mathcal{H}(H_n)$ for $n \geq 3$.}
\label{poset}
\end{center}
\end{figure}

\begin{remark}
As pointed out to us by M. Zaremsky, Theorem~\ref{thm:Intro} agrees with the description of the BNS invariants $\Sigma^1$ of Houghton groups obtained in \cite{MR914847}. This is due to the fact that, as shown in \cite{MR914847}, BNS invariants can be characterised by actions on real trees, providing a subset of focal hyperbolic structures. A consequence of Theorem~\ref{thm:Intro} is that there are no other possible focal hyperbolic structures for Houghton groups. 
\end{remark}

\noindent
In order to prove Theorem~\ref{thm:Intro}, we consider a specific family of subgroups of $H_n$, which we call \emph{partial Houghton subgroups}. Namely, for every $I \subset \{1, \ldots, n\}$, we consider the subgroup $H_n(I) \leq H_n$ given by the transformations that fix pointwise a subray in $R_i$ for every $i \notin I$. Equivalently, $H_n(I)= \bigcap_{i \notin I}\mathrm{ker}(\lambda_i)$. Theorem~\ref{thm:SecondPartialHoughton} describes the focal hyperbolic structures of $H_n(I)$ when $I$ is of size two. This is the key step of our argument. Since $H_n(I)$ decomposes as a semidirect product $\mathfrak{S}_\infty \rtimes \mathbb{Z}$, following \cite{MR4069979} and consecutive works, we can use the machinery of confining automorphisms developed in \cite{MR3420526} in order to investigate the hyperbolic structures of $H_n(I)$. Next, in order to deduce Theorem~\ref{thm:Intro}, the key observation is that, with respect to a focal hyperbolic structure, there always exists some $I \subset \{1, \ldots, n\}$ of size two such that $H_n(I)$ is cobounded in $H_n$.

\paragraph{Groups with few focal hyperbolic structures.} The central problem, in the study of posets of hyperbolic structures, is to understand which posets can be realised as the poset of hyperbolic structures of a group and what it tells us about the group. As a particular case, \cite[Problem~8.5]{MR3995018} asks whether there exist groups having an arbitrary finite number of focal hyperbolic structures. The first examples of groups with only finitely many focal hyperbolic structures appeared in \cite{MR4069979}, as lamplighter groups. Next, additional examples were obtained in \cite{MR4602410, MR4585997}, as (generalised) Baumslag-Solitar groups. According to Theorem~\ref{thm:Intro}, Houghton groups can now be added to the list. However, in all these examples, the number of focal hyperbolic structures is always at least two, leaving the existence of groups having a single focal hyperbolic structure unknown. 

\medskip \noindent
Nevertheless, it turns out that such groups can be constructed thanks to Houghton groups, answering completely \cite[Problem~8.5]{MR3995018}.

\begin{thm}\label{thm:IntroSingle}
For all non-negative integers $n > k+1$, there exists a finite extension $H_n \rtimes F$ of $H_n$ that has exactly $k$ focal hyperbolic structures.
\end{thm}

\noindent
More precisely, identifying the group $\mathrm{QAut}(\mathcal{R}_n)$ of quasi-automorphisms $\mathcal{R}_n \to \mathcal{R}_n$ with $H_n \rtimes \mathrm{Sym}(n)$, where $\mathrm{Sym}(n)$ permutes the $n$ rays of $\mathcal{R}_n$, we consider the finite extensions $H_n(G):= H_n \rtimes G$ for $G \leq \mathrm{Sym}(n)$. Thanks to Theorem~\ref{thm:Intro}, it is possible to describe precisely the structure of the poset of hyperbolic structures of $H_n(G)$. See Theorem~\ref{thm:HoughtonPerm}. In particular, it turns out that $H_n(G)$ has exactly $k$ focal hyperbolic structures, where $k$ is the number of fixed points for the action of $G$ on $\{1, \ldots, n\}$.

\paragraph{Acknowledgements.} We are grateful to S. Balasubramanya, F. Fournier-Facio, and M. Zaremsky for their comments on a preliminary version of our paper.

\section{Preliminaries}\label{section:Preliminaries}

\noindent
In this section, we collect basic definitions and results that we will use in the rest of the article, related to actions on hyperbolic spaces and Houghton groups.

\paragraph{Classification of actions on hyperbolic spaces.} First, we recall some terminology related to a actions on hyperbolic spaces. 

\begin{definition}
Let $G$ be a group acting on a hyperbolic space $X$. The action is:
\begin{itemize}
	\item \emph{bounded} if orbits are bounded;
	\item \emph{horocyclic} if it is unbounded but has no loxodromic element;
	\item \emph{lineal} if it has a loxodromic element and all the loxodromic elements have the same endpoints at infinity;
	\item \emph{focal} if it has a loxodromic element, is not lineal, and any two loxodromic elements have a common endpoint;
	\item \emph{general type} if it has two loxodromic elements with no common endpoints.
\end{itemize}
\end{definition}

\noindent
Since we consider only cobounded actions, horocyclic actions will never occur. Let us record for future use a couple of easy lemmas.

\begin{lemma}\label{lem:CommSubEll}
Let $G$ be a group acting coboundedly on a hyperbolic space $X$. If $[G,G]$ has bounded orbits, then $G \curvearrowright X$ is lineal or bounded. 
\end{lemma}

\begin{proof}
It follows for instance from \cite[Lemma~4.10]{NL} that there exist a hyperbolic space $Y$ and an $\alpha$-equivariant quasi-isometry $X \to Y$ where $\alpha : G \to G/[G,G]$ is the quotient map. Since $G/[G,G]$ is abelian, the action cannot be focal nor general type. The action cannot be horocyclic either since it is cobounded. Therefore, the action must be either bounded or lineal. 
\end{proof}

\noindent
In our next lemma, a group $G$ is said \emph{boundedly generated} by a collection $(H_1, \ldots, H_n)$ of subgroups if $G= H_1 \cdots H_n$. 

\begin{lemma}\label{lem:BoundedOrbits}
Let $G$ be a group acting on a metric space $X$. Assume that $G$ is boundedly generated by some collection $(H_1, \ldots, H_n)$. If each $H_i$ has bounded orbits in $X$, then $G$ has bounded orbits in $X$. 
\end{lemma}

\begin{proof}
Fix a basepoint $o \in X$. We want to prove by induction over $k$ that the $H_1 \cdots H_k$-orbit of $o$ is bounded. For $k=1$, this is true by induction. Assuming that $k \geq 2$ and that $H_1 \cdots H_{k-1} \cdot o$ is bounded, we have
$$d(o,h_1 \cdots h_k \cdot o ) \leq d(o,h_1 \cdots h_{k-1} \cdot o ) + d(o,h_k \cdot o) \leq A+B$$
for all $h_1 \in H_1, \ldots, h_k \in H_k$, where $A$ (resp.\ $B$) denotes the diameter of the $H_1 \cdots H_{k-1}$-orbit of $o$ (resp.\ of the $H_k$-orbit). Therefore, the $H_1 \cdots H_k$-orbit of $o$ has diameter $\leq A+B$.
\end{proof}

\begin{cor}\label{cor:BoundedOrbits}
Let $G$ be a group acting on a metric space $X$. Fix a finite generating set $\{s_1, \ldots, s_n\}$. If $[G,G]$ and each $s_i$ have bounded orbits in $X$, then $G$ has bounded orbits in $X$.
\end{cor}

\begin{proof}
It suffices to notice that $G$ is boundedly generated by $(\langle s_1 \rangle, \ldots, \langle s_n \rangle, [G,G])$ and to apply Lemma~\ref{lem:BoundedOrbits}. 
\end{proof}

\paragraph{Busemann quasimorphisms.} When a group acts on a hyperbolic space with a fixed point at infinity, there is map that records, roughly speaking, the amount with which an element translates the points of the space towards the fixed point at infinity. This map is not too far from being a morphism to $\mathbb{R}$. This property is recorded by the notion of \emph{quasimorphism}, we which we define now. 

\begin{definition}
Let $G$ be a group. A function $q : G \to \mathbb{R}$ is a \emph{quasimorphism} if there exists $D \geq 0$ such that
$$|q(gh) - q(g) - q(h)| \leq D \text{ for all } g,h \in G.$$
It is \emph{homogeneous} if it restricts to a morphism on each cyclic subgroup. Its \emph{homogenisation} is 
$$g \mapsto \lim\limits_{n \to + \infty} \frac{q(g^n)}{n}.$$
A group is \emph{agreeable} if all its homogeneous quasimorphisms are morphisms.
\end{definition}

\noindent
Examples of agreeable groups include amenable groups (see for instance \cite[Proposition~2.65]{MR2527432}), so in particular Houghton groups are agreeable. 

\begin{definition}
Let $G$ be a group acting on a hyperbolic space $X$ with fixed point $\xi \in \partial X$ at infinity. Given a basepoint $o \in X$ and a sequence $(x_i)_{i \geq 0}$ converging to $\xi$, 
$$g \mapsto \limsup\limits_{n \to + \infty} \left( d(o,x_n) - d(go,x_n) \right)$$
is the \emph{Busemann quasimorphism}. 
\end{definition}

\noindent
It is worth mentioning that the Busemann quasimorphism turns out to be independent of the choices of $o$ and $(x_i)_{i \geq 0}$. When its homogenisation is a morphism, the action is \emph{regular}; and then we refer to this homogenisation as the \emph{Busemann morphism}. Notice that a Busemann morphism is non-trivial exactly on the loxodromic elements.

\medskip \noindent
Essentially by definition, a lineal action on a hyperbolic space stabilises a pair of points at infinity, namely the two endpoints common to all the loxodromic elements. We may distinguish two cases: if the group fixes these two points, the action is \emph{oriented}; otherwise, it is \emph{non-oriented}. The typical example of a non-oriented lineal action is the action of $\mathbb{D}_\infty$ on $\mathbb{R}$. In fact, it follows from \cite[Corollary~1.15]{BFFZ} that this is essentially the only example:

\begin{lemma}\label{lem:NonOrientedLineal}
A finitely generated amenable group admits a non-oriented lineal hyperbolic structure if and only if it surjects onto $\mathbb{D}_\infty$. 
\end{lemma}

\noindent
It is not difficult to see that Houghton groups never surject onto $\mathbb{D}_\infty$ (see Fact~\ref{fact:NotOntoDinfty}), so we can focus on the set of oriented lineal actions, which we denote by $\mathcal{H}_\ell^+(\cdot)$. As we want to explain now, oriented lineal actions of a group $G$ are naturally encoded by the \emph{projective space of characters} 
$$\Pi(G):= (H^1(G,\mathbb{R}) \backslash \{0\})/ \mathbb{R}^\ast.$$ 
Recall from \cite[Lemma~4.15]{MR3995018} that, for every non-trivial morphism $\varphi \in H^1(G,\mathbb{R})$, there exists a constant $C_\varphi \geq 0$ such that $\varphi^{-1}([-C,C])$ gives an oriented lineal hyperbolic structure of $G$ for every $C \geq C_\varphi$. Moreover, the structure does not depend on the choice of $C$ and an element $g \in G$ is loxodromic if and only if $\varphi(g) \neq 0$. Then:

\begin{lemma}\label{lem:LinealActions}
Let $G$ be an agreeable group. Then
$$\left\{ \begin{array}{ccc} H^1(G, \mathbb{R}) \backslash \{0\} & \to & \mathcal{H}_\ell^+(G) \\ \varphi & \mapsto & \left[ \varphi^{-1}([-C_\varphi,C_\varphi]) \right] \end{array} \right.$$
induces an $\mathrm{Aut}(G)$-equivariant bijection $\Pi(G) \to \mathcal{H}_\ell^+(G)$. 
\end{lemma}

\noindent
In this statement, we make $\mathrm{Aut}(G)$ act on $\mathcal{H}_\ell^+(G)$ via
$$\varphi \ast [X] := [\varphi(X)], \ \varphi \in \mathrm{Aut}(G), [X] \in \mathcal{H}_\ell^+(G);$$
and we make it act on $\Pi(G)$ via
$$\varphi \ast \psi := \psi \circ \varphi^{-1}, \ \varphi \in \mathrm{Aut}(G), \psi \in H^1 (G, \mathbb{R}).$$
Despite the fact that we did not find this statement in the literature, Lemma~\ref{lem:LinealActions} is well-known to specialists, so we only include a sketch of proof. 

\begin{proof}[Sketch of proof of Lemma~\ref{lem:LinealActions}.]
It is clear from what we already said that we get a Well-defined map $\Phi : \Pi(G) \to \mathcal{H}_\ell^+(G)$. It is also clear that it is $\mathrm{Aut}(G)$-equivariant. What we need to justify that is that this map is a bijection. Surjectivity follows from the fact that $G$ is agreeable, since, for every oriented lineal hyperbolic structure $[X] \in \mathcal{H}_\ell^+(G)$, we have $\Phi(\beta_X) = [X]$ where $\beta_X$ denotes the Busemann morphism given by the action of $G$ on $\mathrm{Cayl}(G,X)$. For the injectivity, if $\varphi_1, \varphi_2 \in H^1(G, \mathbb{R})$ are distinct in $\Pi(G)$, then their kernels must be different. So we can fix an element $g\in G$, say, in the kernel of $\varphi_1$ but not in the kernel of $\varphi_2$. Then $g$ has bounded orbits with respect to $\Phi(\varphi_1)$ but is loxodromic with respect to $\Phi(\varphi_2)$, proving that the two hyperbolic structures $\Phi(\varphi_1)$ and $\Phi(\varphi_2)$ must be different. 
\end{proof}

\paragraph{Houghton groups.} Recall from the introduction that $\mathcal{R}_n$ denotes the disjoint unions of $n$ rays $R_i:= \{i\} \times \mathbb{N}$, $1 \leq i \leq n$; and that the Houghton group $H_n$ is the group of quasi-automorphisms $\mathcal{R}_n \to \mathcal{R}_n$ that restrict eventually to translations on each ray. Clearly, $H_n$ contains the group of all finitely-supported permutations of $\mathcal{R}_n$, which we denote by $\mathfrak{S}_\infty$. For all distinct $1 \leq i,j \leq n$, we denote by $t_{i,j} \in H_n$ the element that fixes pointwise each ray $R_k$, $k \neq i,j$, and that translates the vertices of $R_i \cup R_j$ from $R_i$ to $R_j$. 

\medskip \noindent
Finally, for every $1 \leq i \leq n$, we denote by $\lambda_i : H_n \to \mathbb{Z}$ the morphism that records the eventual algebraic translation length of an element on the ray $R_i$. Here, a positive translation length corresponds to moving the vertices towards the infinity. These morphisms turn out to encode the abelianisation of $H_n$:

\begin{lemma}\label{lem:abelianisation}
Let $n \geq 2$ be an integer. The morphism
$$\lambda_1 \oplus \cdots \oplus \lambda_{n-1} : H_n \to \mathbb{Z}^{n-1}$$
is the abelianisation map of $H_n$. In particular, $[H_n,H_n]= \mathfrak{S}_\infty$. 
\end{lemma}

\begin{proof}
For short, set $\lambda:= \lambda_1 \oplus \cdots \oplus \lambda_{n-1}$. We have
$$\mathrm{ker}(\lambda)= \bigcap\limits_{i=1}^{n-1} \mathrm{ker}(\lambda_i) = \bigcap\limits_{i=1}^n \mathrm{ker}(\lambda_i)= \mathfrak{S}_\infty.$$
Therefore, saying that $\lambda$ is the abelianisation map of $H_n$ amounts to saying that $[H_n,H_n]= \mathfrak{S}_\infty$. In order to justify this equality, see for instance \cite[Lemma~2.3]{LeeHoughton}. 
\end{proof}

\section{Second partial Houghton groups}

\noindent
We start our study of hyperbolic structures on Houghton groups by considering a specific family of subgroups, which we refer as \emph{partial Houghton groups}. 

\begin{definition}
Let $n \geq 1$ be an integer and $I \subset \{1, \ldots, n\}$ a set. The \emph{partial Houghton group} $H_{n}(I)$ is the subgroup of $H_n$ given by the transformations that fix pointwise a subray in $\{i\} \times \mathbb{N}$ for every $i \notin I$. 
\end{definition}

\noindent
It is worth noticing that, isomorphically, the subgroup $H_n(I) \leq H_n$ only depends on the cardinality of $I$. Indeed, the symmetric subgroup of $\mathrm{QAut}(\mathcal{R}_n)$ that permutes the rays of $\mathcal{R}_n$ normalises $H_n$, inducing an action on $H_n$ by automorphisms. As a consequence, if $I_1,I_2 \subset \{1, \ldots, n\}$ are two subsets of the same size, then there exists an automorphism of $H_n$ sending $H_n(I_1)$ to $H_n(I_2)$. Therefore, it makes sense to focus on the subgroups $H_n(r):= H_n(\{1,\ldots, r\})$ where $0 \leq r \leq n$.

\medskip \noindent
In this section, our goal is to describe the possible hyperbolic structure on the second partial Houghton subgroup $H_n(2) \leq H_n$. This will be the key in order to understand the hyperbolic structures of Houghton groups.  

\begin{thm}\label{thm:SecondPartialHoughton}
Let $n \geq 2$ be an integer. The partial Houghton group $H_n(2)$ has exactly two focal hyperbolic structures, namely $[\mathrm{Fix}(R_1) \cup \{ t_{1,2}\}]$ and $[\mathrm{Fix}(R_2) \cup \{ t_{1,2}^{-1}\}]$; they are not comparable and they both dominate the single lineal hyperbolic structure of $H_n(2)$, namely $[\mathfrak{S}_\infty \cup \{ t_{1,2}\}]$. 
\end{thm}

\noindent
In Section~\ref{section:confining}, we start by stating basic definitions and results related to \emph{confining automorphisms}. In Section~\ref{section:PartialFocal}, we describe the possible focal hyperbolic structures on second partial Houghton subgroups, and we conclude the proof of Theorem~\ref{thm:SecondPartialHoughton} in Section~\ref{section:PartialHyp}.

\subsection{Confining subsets}\label{section:confining}

\noindent
Following \cite{MR3420526}, we will use \emph{confining automorphisms} in order to describe focal hyperbolic structures. Recall that:

\begin{definition}
Let $G$ be a group, $Q \subset G$ a subset. An automorphism $\alpha \in \mathrm{Aut}(G)$ is (\emph{strictly}) \emph{confining $G$ into $Q$} if the following conditions are satisfied:
\begin{itemize}
	\item $\alpha(Q)$ is (strictly) contained in $Q$;
	\item $G= \bigcup_{n \geq 0} \alpha^{-n}(Q)$;
	\item there exists $n_0 \geq 0$ such that $\alpha^{n_0}(Q^2) \subset Q$. 
\end{itemize}
\end{definition}

\noindent
The typical example to keep in mind is the following: Let $G$ be a group acting on a tree $T$ with a global fixed point $\xi \in \partial T$ at infinity. Every element of $G$ reduces to a translation on some subray pointing to $\xi$. Let $\lambda : G \to \mathbb{Z}$ denote the morphism that records this eventual algebraic translation length. Assuming that $G$ contains a loxodromic element, $\lambda$ is non-trivial and can be used to decompose $G$ as a semidirect product $\mathrm{ker}(\lambda) \rtimes \mathbb{Z}$. Let $g \in G$ be a generator of the $\mathbb{Z}$-factor. It is necessarily a loxodromic element, so it admits an axis $\gamma$. Fix a subray $\gamma^- \subset \gamma$ that does not point to $\xi$. Up to replacing $g$ with its inverse, assume that $g$ translates the vertices of $\gamma$ towards $\xi$. Then $g$ strictly confines $\mathrm{ker}(\lambda)$ into $\mathrm{Fix}(\gamma^-)$. 

\medskip \noindent
This example generalises to arbitrary hyperbolic spaces, and can be used to described arbitrary regular focal action. More precisely, we have the following statement, which follows from \cite[Theorem~4.1 and Proposition~4.6]{MR3420526}:

\begin{prop}\label{prop:ConfiningToHyp}
Let $G$ be a group and $\alpha \in \mathrm{Aut}(G)$ an automorphism. If $\alpha$ is strictly confining $G$ into some $Q \subset G$, then $Q \cup \{\alpha\}$ defines a regular focal hyperbolic structure for $G \rtimes_\alpha \mathbb{Z}$. Moreover, if $\beta$ denotes the corresponding Busemann morphism, then $\beta(\alpha)>0$. \\ Conversely, assume that $\mathfrak{X}$ is a regular focal hyperbolic structure of a group $H$. There exist an element $h \in H$ and a subset $Q \subset [H,H]$ such that:
\begin{itemize}
	\item $h$ has infinite order in $H/ [H,H]$;
	\item the subgroup $[H,H] \rtimes \langle h \rangle$ is cobounded in $(H,\mathfrak{X})$;
	\item and $h$ is strictly confining $[H,H]$ into $Q$.
\end{itemize} 
\end{prop}

\noindent
It is not difficult to deduce the following statement (which, formally, is also a consequence of the generalisation \cite[Theorem~1.2]{MR4585997}):

\begin{cor}\label{cor:FocalDes}
Let $L \rtimes \mathbb{Z}$ be a semidirect product with $L$ locally finite. Fix a generator $t$ of the $\mathbb{Z}$-factors. The focal hyperbolic structures of $L \rtimes \mathbb{Z}$ are 
$$[Q \cup \{t\}] \text{ (resp.\ $[Q \cup \{t^{-1}\}]$)}$$
where $t$ (resp.\ $t^{-1}$) is strictly confining $L$ into $Q$. 
\end{cor}

\medskip \noindent
Of course, given a semidirect product $H \rtimes_\alpha \mathbb{Z}$, there exists many $Q \subset H$ such that $\alpha$ is strictly confining $H$ into $Q$ and such that $[Q \cup \{ \alpha \}]$ always defines the same hyperbolic structure. In the next two lemmas, we show how we can enlarge this subset $Q$ without modifying the hyperbolic structure. 

\begin{lemma}\label{lem:BoundedIncreasing}
Let $G$ be a group and $\alpha \in \mathrm{Aut}(G)$ an automorphism. Assume that $\alpha$ is confining $G$ into some $Q,S \subset G$. If $S \subset G$ is bounded relative to $\|\cdot\|_Q$, then:
\begin{itemize}
	\item $\alpha$ is also confining $G$ into $Q \cup S$;
	\item and $[Q \cup \{\alpha\}] = [Q \cup S \cup \{\alpha\}]$ for $G \rtimes_\alpha \mathbb{Z}$. 
\end{itemize}
\end{lemma}

\begin{proof}
The second item is clear. In order to prove the first item, fix an integer $n_0 \geq 0$ such that $\alpha^{n_0}(Q^2) \subset Q$. Of course, we have
$$\alpha(Q \cup S) = \alpha(Q) \cup \alpha(S) \subset Q \subset S$$
and 
$$\bigcup\limits_{n \geq 0} \alpha^{-n} (Q \cup S ) = \bigcup\limits_{n \geq 0 } \alpha^{-n}(Q) \cup \bigcup\limits_{n \geq 0} \alpha^{-n}(S)= G.$$
Since $\alpha^n(Q) \subset Q$ and $\alpha^n(S) \subset S$ for every $n \geq n_0$, it remains to verify that $\alpha^{Nn_0}(QS)$ and $\alpha^{Nn_0}(SQ)$ are both contained in $Q$, where $N$ denotes the diameter of $S$ relative to $\|\cdot\|_Q$, in order to conclude that $\alpha$ is confining $G$ into $Q \cup S$. But it follows from Fact~\ref{fact:Bounded} below that
$$\alpha^{Nn_0}(qs) = \alpha^{n_0} \left( \alpha^{(N-1)n_0}(q) \alpha^{(N-1)n_0}(s) \right) \in \alpha^{n_0} \left( \alpha^{(N-1)n_0}(Q^N) Q \right) \subset \alpha^{n_0}(Q^2) \subset Q$$
for all $q \in Q$ and $s \in S$. The inclusion $\alpha^{Nn_0}(SQ) \subset Q$ is obtained similarly. 

\begin{fact}\label{fact:Bounded}
For every $k \geq 1$, $\alpha^{(k-1)n_0}(Q^k) \subset Q$.
\end{fact}

\noindent
We argue by induction over $k$. For $k \leq 2$, the inclusion is clear. Then, assuming that $k > 2$, we have
$$\alpha^{kn_0}(Q^{k+1}) = \alpha^{n_0} \left( \alpha^{(k-1)n_0} (Q^k) \alpha^{(k-1)n_0}(Q) \right) \subset \alpha^{n_0} \left( Q^2 \right) \subset Q,$$
where the first inclusion is justified by our induction hypothesis.
\end{proof}

\begin{lemma}\label{lem:IncreaseingSubset}
Let $G$ be a group, $Q \subset G$ a subset, and $\alpha \in \mathrm{Aut}(G)$ an automorphism (strictly) confining $G$ into $Q$. Fix an $n_0 \geq 0$ such that $\alpha^{n_0}(Q^2) \subset Q$. Let $f_1,f_2, \ldots \in G$ be elements for which there exists $K \geq 1$ such that $\alpha^K(f_i) \in Q$ for every $i \geq 1$. Define
\begin{itemize}
	\item $P_0:= \{ \alpha^i(f_j) \mid i \geq 0, j \geq 1 \} \cup Q$;
	\item and $P_n:= \alpha^{n_0}(P_{n-1}^2) \cup P_{n-1}$ for every $n \geq 1$.
\end{itemize}
Then $\overline{Q}:= \bigcup_{n \geq 0} P_n$ satisfies the following assertions:
\begin{itemize}
	\item $\alpha$ is (strictly) confining $G$ into $\overline{Q}$;
	\item $\alpha^{n_0}( \overline{Q}^2) \subset \overline{Q}$;
	\item $Q \cup \{\alpha\}$ and $\overline{Q} \cup \{\alpha\}$ define the same hyperbolic structure of $G \rtimes_\alpha \mathbb{Z}$. 
\end{itemize}
\end{lemma}

\begin{proof}
It is clear that $\alpha(P_0) \subset P_0$. Then, notice that, if $n \geq 1$ is an integer such that $\alpha(P_{n-1}) \subset P_{n-1}$, then
$$\alpha(P_n) = \alpha^{n_0}\left( \alpha(P_{n-1})^2 \right) \cup \alpha(Q) \subset \alpha^{n_0} \left( P_{n-1}^2 \right) \cup Q = P_n.$$
Therefore, it follows by induction that $\alpha(P_n) \subset P_n$ for every $n \geq 0$, which implies that $\alpha(\overline{Q}) \subset \overline{Q}$. Next, take two elements $a,b \in \overline{Q}$. We can fix an integer $n \geq 0$ such that $a$ and $b$ both belong to $P_n$. Hence
$$\alpha^{n_0}(ab) \in \alpha^{n_0}(P_n^2) \subset P_{n+1} \subset \overline{Q}.$$
Thus, we know that $\alpha^{n_0}(\overline{Q}^2) \subset \overline{Q}$. Finally, since $Q \subset P_0 \subset \overline{Q}$, 
$$G = \bigcup\limits_{n \geq 0} \alpha^{-n}(Q) \subset \bigcup\limits_{n \geq 0} \alpha^{-n}( \overline{Q}), \text{ hence } \bigcup\limits_{n \geq 0} \alpha^{-n}( \overline{Q})=G.$$
So far, we have proved that $\alpha$ is confining $G$ into $\overline{Q}$ and we have verified the second item of our lemma. In order to verify the third item, notice that $\alpha^K(P_0) \subset Q$; and that, if $n \geq 1$ is an integer such that $\alpha^K(P_{n-1}) \subset Q$, then 
$$\alpha^K(P_n)  \alpha^{n_0} \left( \alpha^K(P_{n-1})^2 \right) \cup \alpha^K(P_{n-1}) \subset \alpha^{n_0} \left( Q^2 \right) \cup Q \subset Q.$$
Therefore, it follows by induction that $\alpha^K(P_n) \subset Q$ for every $n \geq 0$, which implies that $\alpha^K(\overline{Q}) \subset Q$. Hence
$$\sup\limits_{s \in \overline{Q} \cup \{\alpha\}} \|s\|_{Q \cup \{ \alpha \}} \leq 2K+1.$$
The reverse inequality
$$\sum\limits_{s \in Q \cup \{\alpha\}} \| s\|_{\overline{Q} \cup \{\alpha\}} \leq 1$$
is clear since $Q \subset P_0 \subset \overline{Q}$. We conclude, as desired, that $Q \cup \{\alpha\}$ and $\overline{Q} \cup \{\alpha\}$ define the same hyperbolic structure of $G \rtimes_\alpha \mathbb{Z}$.

\medskip \noindent
Finally, since we have seen that $\alpha^K(\overline{Q}) \subset Q \subset \overline{Q}$, we know that
$$\alpha^{2K} (Q) \subset \alpha^{2K}(\overline{Q}) \subset \alpha^K(Q) \subset \alpha^K(\overline{Q}) \subset Q \subset \overline{Q}.$$
We conclude that $\alpha(Q)=Q$ if and only if $\alpha(\overline{Q})= \overline{Q}$. In other words, $\alpha$ is strictly confining $G$ into $Q$ if and only if it is strictly confining $G$ into $\overline{Q}$. 
\end{proof}

\subsection{Possible focal actions}\label{section:PartialFocal}

\noindent
In this subsection, our goal is to describe the possible focal hyperbolic structures on the second partial Houghton subgroup $H_n(2) \leq H_n$. Namely:

\begin{thm}\label{thm:FocalTwo}
Let $n \geq 2$ be an integer. The partial Houghton group $H_n(2)$ admits exactly two focal hyperbolic structures, namely $[\mathrm{Fix}(R_1) \cup  \{t_{1,2}\}]$ and $[\mathrm{Fix}(R_2) \cup \{t_{1,2}^{-1}\}]$.
\end{thm}

\noindent
For convenience, in the sequel we will use the following notations. First, we will set $t:=t_{1,2}$. Since $t$ normalises the subgroup $\mathfrak{S}_\infty \leq H_n$, it induces an automorphism on $\mathfrak{S}_\infty$, which we sill denote by $\tau$. Next, we will identify $R_1 \cup R_2$ with $\mathbb{Z}$ via 
$$\left\{ \begin{array}{ccc} \{1\} \times \mathbb{N} & \to & (- \infty,0) \\ (1,i) & \mapsto & -i-1 \end{array} \right. \text{ and } \left\{ \begin{array}{ccc} \{2\} \times \mathbb{N} & \to & [0,+ \infty) \\ (2,i) & \mapsto & i \end{array} \right.$$
In particular, $R_1 \cup R_2$ will be totally ordered thanks to this identification.

\medskip \noindent
We start by verifying that the subsets given by Theorem~\ref{thm:FocalTwo} are indeed focal hyperbolic structures. 

\begin{lemma}\label{lem:SecondNotComparable}
The generating sets $\mathrm{Fix}(R_1) \cup  \{t\}$ and $\mathrm{Fix}(R_2) \cup \{t^{-1}\}$ of $H_n(2)$ induce two non-comparable focal hyperbolic structures.
\end{lemma}

\begin{proof}
In order to verify that $\mathrm{Fix}(R_1) \cup  \{t\}$ and $\mathrm{Fix}(R_2) \cup \{t^{-1}\}$ induce two focal hyperbolic structures on $H_n(2)$, according to Proposition~\ref{prop:ConfiningToHyp} it suffices to show that:

\begin{claim}\label{claim:FixCofining}
The automorphism $\tau$ (resp.\ $\tau^{-1}$) is strictly confining $\mathfrak{S}_\infty$ into $\mathrm{Fix}(R_1)$ (resp.\ $\mathrm{Fix}(R_2)$).
\end{claim}

\noindent
It is clear that $\tau(\mathrm{Fix}(R_1)) = \mathrm{Fix}((- \infty,1])\subsetneq \mathrm{Fix}(R_1)$ and that $\tau^0(\mathrm{Fix}(R_1)^2) = \mathrm{Fix}(R_1)$. Next, if $\sigma \in \mathfrak{S}_\infty$ is an arbitrary permutation, then we can fix a sufficiently large integer $p \geq 1$ such that $\sigma$ fixes $(- \infty,-p]$ pointwise and we have $\tau^p(\sigma) \in \mathrm{Fix}(R_1)$. Therefore, $\bigcup_{n \geq 1} \tau^{-n}(\mathrm{Fix}(R_1)) = \mathfrak{S}_\infty$. We conclude that $\tau$ is strictly confining $\mathfrak{S}_\infty$ into $\mathrm{Fix}(R_1)$. The corresponding statement for $\tau^{-1}$ is obtained by symmetry under the automorphism of $H_n(2)$ that inverts $\tau$. This concludes the proof of Claim~\ref{claim:FixCofining}. 

\medskip \noindent
It remains to verify that the structures $[\mathrm{Fix}(R_1) \cup  \{t\}]$ and $[\mathrm{Fix}(R_2) \cup \{t^{-1}\}]$ are not comparable. As a consequence of Fact~\ref{fact:Norm} below, 
$$\| (-n, -n+1) \|_{\mathrm{Fix}(R_1) \cup \{t\}} \underset{n \to + \infty}{\longrightarrow} + \infty \text{ but } \| (-n,-n+1) \|_{\mathrm{Fix}(R_2) \cup \{t^{-1}\}} = 1.$$
By symmetry under the automorphism of $H_n(2)$ that inverts $\tau$, we have similarly
$$\| (n, n+1) \|_{\mathrm{Fix}(R_2) \cup \{t^{-1}\}} \underset{n \to + \infty}{\longrightarrow} + \infty \text{ but } \| (n,n+1) \|_{\mathrm{Fix}(R_1) \cup \{t\}} = 1.$$
This concludes the proof of our lemma.

\begin{fact}\label{fact:Norm}
The equality
$$\|\sigma\|_{\mathrm{Fix}(R_1) \cup \{t\}} = 1+2 \cdot \min \{ k \geq 0 \mid \mathrm{supp}(\sigma) \cap (- \infty, -k)= \emptyset \}$$
holds for every non-trivial $\sigma \in \mathfrak{S}_\infty$.
\end{fact}

\noindent
For every permutation $\sigma \in \mathfrak{S}_\infty$, we set
$$k(\sigma):=  \min \{ k \geq 0 \mid \mathrm{supp}(\sigma) \cap (- \infty, -k)= \emptyset \}.$$
For short, we also write $\| \cdot\|$ for $\|\cdot\|_{\mathrm{Fix}(R_1) \cup \{t\}}$. Given a non-trivial permutation $\sigma \in \mathfrak{S}_\infty$, notice that $\tau^{k(\sigma)}(\sigma)$ fixes $R_1$ pointwise. From the equality
$$\sigma = t^{-k(\sigma)} \cdot \tau^{k(\sigma)}(\sigma) \cdot t^{k(\sigma)},$$
it follows that $\|\sigma\| \leq 1+ 2k (\sigma)$. Next, decompose $\sigma$ as a word of minimal length over $\mathrm{Fix}(R_1) \cup \{t\}$, i.e.\ 
$$\sigma = t^{\epsilon_1} \sigma_1 t^{\epsilon_2}\sigma_2 \cdots t^{\epsilon_n} \sigma_n t^{\epsilon_{n+1}}$$
where $\sigma_1, \ldots, \sigma_n \in \mathrm{Fix}(R_1) \backslash \{1\}$, $n \geq 1$, $\epsilon_2, \ldots, \epsilon_n \in \{-1, 1\}$, and $\epsilon_1, \epsilon_{n+1} \in \{ 0, -1, 1\}$. Notice that
$$\|\sigma\| = n + \sum\limits_{i=1}^{n+1} |\epsilon_i|.$$
We can write
$$\sigma = t^{\epsilon_1}\sigma_1t^{-\epsilon_1} \cdot t^{\epsilon_1+\epsilon_2} \sigma_2 t^{-\epsilon_1- \epsilon_2} \cdots t^{\epsilon_1 + \cdots + \epsilon_n} \sigma_n t^{-\epsilon_1- \cdots- \epsilon_n} \cdot t^{\epsilon_1+\cdots+ \epsilon_{n+1}}.$$
Because $\sigma$ belongs to $\mathfrak{S}_\infty$, necessarily $\epsilon_1+ \cdots + \epsilon_{n+1}=0$, and we have
$$\sigma = t^{\epsilon_1}\sigma_1t^{-\epsilon_1} \cdot t^{\epsilon_1+\epsilon_2} \sigma_2 t^{-\epsilon_1- \epsilon_2} \cdots t^{\epsilon_1 + \cdots + \epsilon_n} \sigma_n t^{-\epsilon_1- \cdots- \epsilon_n} .$$
Notice that, for every $1 \leq i \leq n$, $t^{\epsilon_1+ \cdots + \epsilon_i} \sigma_i t^{- \epsilon_1- \cdots - \epsilon_i}$ fixes $(- \infty, -| \epsilon_1+ \cdots + \epsilon_i|)$ pointwise. A fortiori, if $1 \leq r \leq n$ satisfies 
$$|\epsilon_1 + \cdots + \epsilon_r| = \max \{ |\epsilon_1+ \cdots + \epsilon_i|, \ 1 \leq i \leq n \},$$
then $\sigma$ fixes $(- \infty, -| \epsilon_1 + \cdots + \epsilon_r|)$ pointwise, hence $k(\sigma) \leq |\epsilon_1 + \cdots + \epsilon_r|$. Using the fact that $\epsilon_1 + \cdots + \epsilon_{n+1}=0$, we deduce that
$$\begin{array}{lcl} 2 k(\sigma) & \leq & 2 |\epsilon_1+ \cdots + \epsilon_r| = |\epsilon_1+ \cdots + \epsilon_r| + |\epsilon_{r+1} + \cdots + \epsilon_{n+1}| \\ \\ & \leq  & \displaystyle \sum\limits_{i=1}^{n+1} |\epsilon_i| = \|\sigma\| -n \leq \| \sigma \| - 1. \end{array}$$
Thus, we have proved that $\|\sigma\| \geq 1+2 k(\sigma)$, concluding the proof of Fact~\ref{fact:Norm}. 
\end{proof}

\noindent
Thanks to Corollary~\ref{cor:FocalDes}, the description of the focal hyperbolic structures on $H_n(2)$ essentially reduces to the following statement:

\begin{prop}\label{prop:StrictCofining}
Assume that $\tau$ is strictly confining $\mathfrak{S}_\infty$ into some $Q \subset \mathfrak{S}_\infty$. Then $[Q \cup \{t\}] = [\mathrm{Fix}(R_1) \cup \{t\}]$. 
\end{prop}

\noindent
In order to prove Proposition~\ref{prop:StrictCofining}, we start by proving the following lemma:

\begin{lemma}\label{lem:Cofining}
Assume that $\tau$ is confining $\mathfrak{S}_\infty$ into some $Q \subset \mathfrak{S}_\infty$. There exists $Q^+ \subset \mathfrak{S}_\infty$ such that 
\begin{itemize}
	\item $\tau$ is confining $\mathfrak{S}_\infty$ into $Q^+$;
	\item $[Q \cup \{t\}]= [Q^+ \cup \{t\}]$;
	\item $Q^+$ contains $Q \cup \mathrm{Fix}(R_1)$. 
\end{itemize}
\end{lemma}

\begin{proof}
Fix an $n_0 \geq 0$ such that $\tau^{n_0}(Q^2) \subset Q$. Our goal is to apply Lemma~\ref{lem:IncreaseingSubset} iteratively in order to increase the size of $Q$ until it contains $\mathrm{Fix}(R_1)$ entirely.

\medskip \noindent
First, let $S$ denote the set of the products of $\leq 4n_0+1$ transpositions of two consecutive vertices in $R_2$. Notice that
$$\tau^{(4n_0+1)n_0+K}(\sigma) \in Q \text{ for every } \sigma \in S$$
where $K \geq 0$ is such that $\tau^K(0,1) \in Q$. Indeed, we can write $\sigma$ as a product $\sigma_1 \cdots \sigma_k$ where $k\leq 4n_0+1$ and where each $\sigma_i$ is a transposition of two consecutive vertices of $R_2$. Since each $\sigma_i$ is conjugate to $(0,1)$ by a positive power of $t$, say $\sigma_i= \tau^{r_i}(0,1)$ for some $r_i \geq 0$, we know that 
$$\tau^K(\sigma_i) = \tau^{r_i} \tau^K(0,1) \in \tau^{r_i}(Q) \subset Q.$$
Thus,
$$\begin{array}{lcl} \tau^{(4n_0+1)n_0+K}(\sigma) & = & \tau^{(4n_0+1-k)n_0} \left( \tau^{kn_0} \left( \tau^K(\sigma_1) \cdots \tau^K(\sigma_k) \right) \right) \\ \\ & \subset & \tau^{(4n_0+1-k)n_0} \left( \tau^{kn_0} \left( Q^k \right) \right) \\ \\ & \subset & \tau^{(4n_0+1-k)n_0}(Q) \subset Q, \end{array}$$
as desired. Thus, we can apply Lemma~\ref{lem:IncreaseingSubset} to $Q$ and $S$. Let $Q'$ denote the subset of $\mathfrak{S}_\infty$ this yields. Notice that:

\begin{claim}\label{claim:Transpositions}
For all $x,y \in R_2$ satisfying $y \geq x \geq 2n_0$, $(x,y)$ belongs to~$Q'$. 
\end{claim}

\noindent
We argue by induction over $y-x$. If $y-x \leq 2n_0$, then $(x,y)$ is a product of less than $4n_0+1$ consecutive transpositions, so we already know that it belongs to $Q'$. From now on, assume that $y-x>2n_0$. We can write
$$(x,y)= (x,x+2n_0) (x+2n_0,y)(x,x+2n_0) = \tau^{2n_0} (x-2n_0,x)(x,y-2n_0)(x-2n_0, x).$$
Notice that $(x-2n_0,x)$ belongs to $Q'$ since it can be written as a product of less than $4n_0+1$ consecutive transpositions. Moreover, we know that $(x,y-2n_0)$ belongs to $Q'$ by induction. Therefore,
$$(x,y) \in \tau^{2n_0} (Q'^3)  = \tau^{n_0} \left( \tau^{n_0}(Q'^2) \tau^{n_0}(Q') \right) \subset \tau^{n_0}(Q'^2) \subset Q'.$$
This concludes the proof of Claim~\ref{claim:Transpositions}.  

\medskip \noindent
Let $S'$ denote the set of all the products of $\leq n_0$ transpositions between two vertices of $R_2$. Notice that, if $\sigma$ is a transposition between two vertices of $R_2$, then $\tau^{2n_0}(\sigma)$ is a transposition supported on $[2n_0,+ \infty)$, so it belongs to $Q'$ according to Claim~\ref{claim:Transpositions}. Now, if $\sigma_1 \cdots \sigma_k$ is product of $k \leq n_0$ transpositions $\sigma_1, \ldots, \sigma_k$ between two vertices of $R_2$, then
$$\tau^{(k+2)n_0} (\sigma_1 \cdots \sigma_k) = \tau^{kn_0} \left( \tau^{2n_0}(\sigma_1) \cdots \tau^{2n_0} (\sigma_k) \right) \in \tau^{kn_0}(Q'^k) \subset Q'.$$
Thus, we can apply Lemma~\ref{lem:IncreaseingSubset} to $Q'$ and $S'$. Let $Q''$ denote the subset of $\mathfrak{S}_\infty$ this yields. Notice that:

\begin{claim}\label{claim:Cycles}
Every cycle supported in $[n_0,+ \infty)$ belongs to $Q''$.
\end{claim}

\noindent
We argue by induction over the length of our cycle $\sigma= (x(0),\ldots, x(n))$, where $x(i) \geq n_0$ for every $0 \leq i \leq n$. If $n\leq 1$, we already know that $\sigma$ belongs to $Q''$. From now on, we assume that $n \geq 2$. Let $x(k_1)< \cdots < x(k_{n_0})$ denote the $n_0$ smallest elements of $\{x(0), \ldots, x(n) \}$. Necessarily, each $x(i)$ distinct from the $x(k_j)$ must belong to $[2n_0,+ \infty)$. Notice that
$$(x(0),\ldots, x(n)) = \left( \prod\limits_{i=1}^{n_0} (x(k_i),x(k_i+1)) \right) \left( x(0), \ldots, \widehat{x(k_j)}, \ldots, x(n) \right)$$
$$= \tau^{n_0} \left( \left( \prod\limits_{i=1}^{n_0} (x(k_i) -n_0 ,x(k_i+1) -n_0) \right) \left( x(0)-n_0, \ldots, \widehat{x(k_j)-n_0}, \ldots, x(n) -n_0\right)\right),$$
where $\left( x(0), \ldots, \widehat{x(k_j)}, \ldots, x(n) \right)$ denotes the cycle obtained from $(x(0), \ldots, x(n))$ by omitting all the $x(k_j)$. We know by induction that 
$$\left( x(0)-n_0, \ldots, \widehat{x(k_j)-n_0}, \ldots, x(n) -n_0\right)$$ 
belongs to $Q''$, hence
$$\sigma \in \tau^{n_0}(S'Q'') \subset \tau^{n_0}(Q''^2) \subset Q''.$$
This concludes the proof of Claim~\ref{claim:Cycles}.

\medskip \noindent
Let $S''$ denote the set of all the products of $\leq n_0$ cycles with pairwise disjoint supports in $R_2$. Notice that, if $\sigma$ is a cycle supported in $R_2$, then $\tau^{n_0}(\sigma)$ is supported in $[n_0,+ \infty)$, and consequently it belongs to $Q''$ according to Claim~\ref{claim:Cycles}. Then, if $\sigma_1 \cdots \sigma_k$ is a product of $k \leq n_0$ cycles $\sigma_1, \ldots, \sigma_k$ with pairwise disjoint supports in $R_2$, then
$$\tau^{(k+1)n_0}(\sigma_1 \cdots \sigma_k) = \tau^{kn_0} \left( \tau^{n_0}(\sigma_1) \cdots \tau^{n_0} (\sigma_k) \right) \in \tau^{kn_0}(Q''^k) \subset Q''.$$
Thus, Lemma~\ref{lem:IncreaseingSubset} can be applied to $Q''$ and $S''$. Let $Q'''$ denote the subset of $\mathfrak{S}_\infty$ this yields. 

\begin{claim}\label{claim:ProductCycles}
Every product of cycles with pairwise disjoint supports in $[n_0,+ \infty)$ belongs to $Q'''$. 
\end{claim}

\noindent
We argue by induction on the number of cycles in our product $\alpha_0 \cdots \alpha_n$. We already know that the desired conclusion holds if $n <n_0$. From now on, we assume that $n\geq n_0$. Since at most $n_0$ cycles may have their supports intersecting $[n_0,2n_0]$, we can assume, up to reordering our product of pairwise commuting cycles, that $\alpha_{n_0}, \ldots, \alpha_n$ have their supports contained $[2n_0,+ \infty)$. Since $\tau^{-n_0}(\alpha_{n_0} \cdot \alpha_n)$ belongs to $Q'''$ by induction and since $\tau^{-n_0}(\alpha_0 \cdots \alpha_{n_0-1})$ belongs to $S''$, we deduce that 
$$\alpha_0 \cdots \alpha_n = \tau^{n_0} \left( \tau^{-n_0}\left( \alpha_0 \cdots \alpha_{n_0-1} \right) \left( \tau^{-n_0} \left( \alpha_{n_0} \cdots \alpha_n \right)\right) \right)$$
belongs to $\tau^{n_0}(S'' Q''') \subset \tau^{n_0}(Q'''^2) \subset Q'''.$
This concludes the proof of Claim~\ref{claim:ProductCycles}. 

\medskip \noindent
Finally, given an arbitrary permutation $\sigma$ supported on $R_2$, $\tau^{n_0}(\sigma)$ decomposes as a product of cycles with pairwise disjoint supports in $[n_0,+ \infty)$, hence $\tau^{n_0}(\sigma) \in Q'''$ according to Claim~\ref{claim:ProductCycles}. Thus, Lemma~\ref{lem:IncreaseingSubset} applies to $Q'''$ and the permutations supported on $R_2$. The subset $Q^+_0$ of $\mathfrak{S}_\infty$ it yields contains all the permutations supported on $R_2$. 

\medskip \noindent
Let us verify that $Q^+:=Q_0^+ \cup \mathrm{Fix}(R_1)$ is the subset we are looking for. According to Lemma~\ref{lem:BoundedIncreasing} and Claim~\ref{claim:FixCofining}, it suffices to verify that $\mathrm{Fix}(R_1)$ is bounded relative to $\| \cdot\|_{Q_0^+}$. So let $\sigma \in \mathrm{Fix}(R_1)$ be a permutation. Fix an integer $M \geq 0$ such that $\mathrm{supp}(\sigma) \cap R_2 \subset \{2\} \times [1,M]$. Because $\tau$ is confining $\mathscr{S}_\infty$ into $Q_0^+$, we know that there exists some $k$ sufficiently large, which we choose larger than $M$, such that $\tau^k(\sigma) \in Q_0^+$. Considering the product
$$\nu := \prod\limits_{i=1}^M (i, i+k)$$
of transpositions supported in $R_2$, notice that $\nu t^k$ fixes $1, \ldots, M$. Therefore, since $\sigma$ fixes $R_1$ pointwise, it follows that $\nu t^k$ and $\sigma$ have disjoint supports. Since $\nu$ belongs to $Q_0^+$, like any permutation supported on $R_2$, it follows that
$$\sigma = \nu \cdot \tau^k(\sigma) \cdot \nu \in (Q_0^+)^3.$$
Thus, Lemma~\ref{lem:BoundedIncreasing} applies, as claimed, proving that $Q^+$ is the subset we were looking for. 
\end{proof}

\begin{proof}[Proof of Proposition~\ref{prop:StrictCofining}.]
Assume that $\tau$ is confining $\mathfrak{S}_\infty$ into some $Q \subset \mathfrak{S}_\infty$. According to Lemma~\ref{lem:Cofining}, we can assume that $\mathrm{Fix}(R_1) \subset Q$, hence $[\mathrm{Fix}(R_1) \cup \{t\} ] \leq [Q \cup \{t\}]$. From now on, we assume that $[\mathrm{Fix}(R_1) \cup \{t\} ] < [Q \cup \{t\}]$. Our goal is to show that necessarily $Q=\mathfrak{S}_\infty$, proving that $\tau$ is not strictly confining $\mathfrak{S}_\infty$ into $Q$. 

\medskip \noindent
Since $[\mathrm{Fix}(R_1) \cup \{t\} ] < [Q \cup \{t\}]$, we know that we find a sequence of elements in $Q$ that is not bounded with respect to $\| \cdot\|_{\mathrm{Fix}(R_1) \cup \{t\}}$. Starting from such a sequence, we will construct new sequences in $Q$ that will look more and more like a transposition between two consecutive vertices in $R_1$. Eventually, we will be able to conclude that every permutation supported on $R_1$ must belong to $Q$, which will essentially proves our proposition since we already know that $Q$ contains $\mathrm{Fix}(R_1)$. For clarity, we decompose our successive constructions into a series of claims.

\medskip \noindent
For the rest of the proof, we fix an integer $n_0 \geq 0$ such that $\tau^{n_0}(Q^2) \subset Q$. 

\begin{claim}\label{claim:Gamma}
For every $n \geq 1$, there exists $\gamma_n \in Q$ fixing $(- \infty,-n)$ pointwise and satisfying $\gamma_n(-n) \neq -n$ but $\gamma_n(-n) \in R_1 \cup R_2$.
\end{claim}

\noindent
Because $[\mathrm{Fix}(R_1) \cup \{t\} ] < [Q \cup \{t\}]$, we know that there exist $\sigma_n \in Q$ satisfying
$$\| \sigma_n\|_{\mathrm{Fix}(R_1) \cup \{t\} } \geq 1+2n \text{ for every } n \geq 1.$$
For every $n \geq 1$, let $k_n \geq 0$ be the smallest integer such that $\mathrm{supp}(\sigma_n) \cap (- \infty, -k_n) = \emptyset$. As a consequence of Fact~\ref{fact:Norm}, $k_n \geq n$. Hence
$$\xi_n:= \tau^{k_n-n}(\sigma_n) \in \tau^{k_n-n}(Q) \subset Q.$$
Notice that, since $\sigma_n(-k_n) \neq -k_n$, necessarily $\xi_n(-n) \neq -n$. However,  $\xi_n(-n)$ may not belong to $R_1 \cup R_2$. Let $\omega_n \in \mathrm{Fix}(R_1)$ be such that $\omega_n(\xi_n(-n)) \in R_1 \cup R_2$. Notice that, necessarily, $\omega_n(\xi_n(-n)) \neq -n$. For every $n \geq 1$, we set
$$\gamma_n:= \tau^{n_0}(\omega_{n+n_0} \circ \gamma_{n+n_0}).$$
Notice that $\gamma_n \in \tau^{n_0}(Q^2) \subset Q$. Also, since $\gamma_{n+n_0}$ and $\omega_{n+n_0}$ fix $(- \infty, -n-n_0)$ pointwise, necessarily $\gamma_n$ fixes $(-\infty,-n)$ pointwise. Moreover, we know by construction that $\gamma_n(-n)$ belongs to $R_1 \cup R_2$ but is distinct from $-n$. This concludes the proof of Claim~\ref{claim:Gamma}. 

\begin{claim}\label{claim:Delta}
For every $n \geq 1$, there exists $\delta_n \in Q$ fixing $(-\infty,-n)$ pointwise and satisfying $\delta_n(-n) \in R_2$.
\end{claim}

\noindent
Fix an $m \geq 1$ and assume that $k:=\gamma_m(-m) <0$, where $\gamma_m$ is the element of $Q$ given by Claim~\ref{claim:Gamma}. Notice that, because $\gamma_m$ fixes $(- \infty,-m)$ pointwise but does not fix $-m$, necessarily $\gamma_m(-m)>m$. Therefore, $1 \leq k <m$. For every $p \geq 0$, we define
$$\Pi(\gamma_m,p) := \tau^{(2^p-1)(m-k)}(\gamma_m) \circ \tau^{(2^p-2)(m-k)} (\gamma_m) \circ \cdots \circ \tau^{m-k}(\gamma_m) \circ \gamma_m.$$
Let us verify that the following assertions are satisfied:
\begin{itemize}
	\item[(i)] $\Pi(\gamma_m,p+1) = \tau^{2^p(m-k)} \left( \Pi(\gamma_m,p) \right) \circ \Pi(\gamma_m,p)$ for every $p \geq 0$.
	\item[(ii)] $\Pi(\gamma_m,p)(-m) = (2^p-1)(m-k)-k \geq 2^p-m$ for every $p \geq 0$. 
	\item[(iii)] $\tau^s \left( \Pi(\gamma_m,p) \right) \in Q$ for all $p \geq 0$ and $s \geq pn_0$.
	\item[(iv)] $\Pi(\gamma_m,p)$ fixes $(- \infty, -m)$ pointwise.
\end{itemize}
In order to prove $(i)$, just notice that
$$\begin{array}{lcl} \Pi(\gamma_m,p+1) & = & \tau^{(2^{p+1}-1)(m-k)} (\gamma_m) \circ \cdots \circ \tau^{(2^{p+1}-2^p)(m-k)}(\gamma_m) \circ \Pi(\gamma_m,p) \\ \\ & = & \tau^{2^p(m-k)} \left( \tau^{(2^p-1)(m-k)}(\gamma_m) \circ \cdots \circ \gamma_m \right) \circ \Pi(\gamma_m,p) \\ \\ & =& \tau^{2^p(m-k)} \left( \Pi(\gamma_m,p) \right) \circ \Pi(\gamma_m,p) \end{array}$$
Thanks to $(i)$, we can prove the equality from $(ii)$ by induction over $p$. For $p=0$, it is clear that $\Pi(\gamma_m,0)(-m) = \gamma_m(-m)=-k$. For $p \geq 1$, we write
$$\begin{array}{lcl} \Pi(\gamma_m,p)(-m) & = & \tau^{2^{p-1}(m-k)}\left( \Pi(\gamma_m,p-1) \right) \circ \Pi(\gamma_m,p-1)(-m) \\ \\ & =& \tau^{2^{p-1}(m-k)}\left( \Pi(\gamma_m,p-1) \right) \left( (2^{p-1}-1)(m-k)-k \right) \\ \\ & =& 2^{p-1}(m-k) + \Pi(\gamma_m,p-1)(-m) \\ \\ & = & 2^{p-1}(m-k) + (2^{p-1}-1)(m-k)-k = (2^p-1)(m-k)-k \end{array}$$
For the inequality, notice that
$$(2^p-1)(m-k)-k  = 2^p(m-k) -m \geq 2^p-m$$
since, as already said, we know that $k<m$. The assertion $(iii)$ can also be proved by induction over $p$ thanks to $(i)$. For $p=0$, 
$$\tau^s \left( \Pi(\gamma_m,0) \right) = \tau^s(\gamma_m) \in \tau^s(Q) \subset Q \text{ for every } s \geq 0.$$
Assuming $p \geq 1$, we have
$$\begin{array}{lcl} \tau^s \left( \Pi(\gamma_m,p) \right) & = & \tau^{n_0} \left[ \tau^{2^{p-1}(m-k)} \left[ \tau^{s-n_0} \left( \Pi(\gamma_m,p-1) \right) \right] \circ \tau^{s-n_0} \left( \Pi(\gamma_m,p-1) \right) \right] \\ \\ & \in & \tau^{n_0} \left( \tau^{2^{p-1}(m-k)}(Q)Q \right) \subset \tau^{n_0}(Q^2) \subset Q \end{array}$$
for every $s \geq pn_0$. Finally, in order to to prove the assertion $(iv)$, it suffices to notice that, for every $i \geq 0$, $\tau^{i(m-k)}(\gamma_m)$ fixes $(- \infty,-m)$ pointwise. Indeed, for every $j<-m$, we have
$$\tau^{i(m-k)}(\gamma_m)(j) = i(m-k)+ \gamma_m\left(j-i(m-k) \right) = i(m-k)+ j -i(m-k)=j.$$
Now, given an $n \geq 1$, we are ready to define the element $\delta_n$ we are looking for. Fix the smallest integer $m$ such that $m-n_0 \lceil \log_2(m) \rceil \geq n$, and set
$$\delta_n:= \left\{ \begin{array}{cl} \tau^{m-n}(\gamma_m) & \text{if } \gamma_m(-m) \in R_2 \\ \tau^{m-n} \left( \Pi(\gamma_m,\lceil \log_2(m) \rceil) \right) & \text{if } \gamma_m(-m) \notin R_2 \end{array} \right..$$
If $\gamma_m(-m) \in R_2$, then $\delta_n \in \tau^{m-n}(Q) \subset Q$; $\delta_n$ fixes $(-\infty, -n)$ pointwise since $\gamma_m$ fixes $(-\infty,-m)$ pointwise; and 
$$\delta_n(-n) = m-n + \gamma_m(-m) \geq m-n \geq 0.$$
Otherwise, if $\gamma_m(-m) \notin R_2$, then $\delta_n \in Q$ according to $(iii)$; $\delta_n$ fixes $(- \infty,-n)$ pointwise since $\Pi(\gamma_m, \lceil \log_2(m) \rceil)$ fixes $(- \infty, -m)$ pointwise according to $(iv)$; and
$$\delta_n(-n) = m-n + \Pi(\gamma_m, \lceil \log_2(m) \rceil)(-m) \geq m-n +2^{\lceil \log_2(m) \rceil} -m \geq 0$$
as a consequence of $(ii)$. This concludes the proof of Claim~\ref{claim:Delta}. 

\begin{claim}\label{claim:Epsilon}
For every $n \geq 1$, there exists $\epsilon_n \in Q$ fixing $(-\infty,-n)$ pointwise and satisfying $\epsilon_n(-n), \epsilon_n^{-1}(-n) \in R_2$.
\end{claim}

\noindent
The strategy is to start from the elements $\delta_n$ given by Claim~\ref{claim:Delta} and to apply the same arguments as in Claims~\ref{claim:Gamma} and~\ref{claim:Delta} in order to move $\delta_n^{-1}(-n)$ inside $R_2$. We start by constructing, for every $n \geq 1$, an element $\pi_n \in Q$ fixing $(-\infty,-n)$ pointwise and satisfying $\pi_n(-n) \in R_2$ and $\pi_n^{-1}(-n) \in R_1 \cup R_2$. 

\medskip \noindent
For every $n \geq 1$, let $\omega_n$ be the identity if $\delta_n^{-1}(-n) \in R_1 \cup R_2$ or a transposition $(\delta_n^{-1}(-n), q_n)$ otherwise where $q_n$ is a large integer than does not belong to $\mathrm{supp}(\delta_n)$. Notice that, since $\omega_n$ fixes $R_1$ pointwise, it belongs to $Q$. We set
$$\pi_n:= \tau^{n_0} \left( \delta_{n+n_0} \circ \omega_{n+n_0} \right) \text{ for every } n \geq 1.$$
It is clear that $\pi_n \in \tau^{n_0}(Q^2) \subset Q$, and that $\pi_n$ fixes $(- \infty,-n)$ pointwise since $\delta_{n+n_0}$ and $\omega_{n+n_0}$ both fix $(- \infty, -n-n_0)$ pointwise. Moreover, 
$$\pi_n(-n) = n_0 + \delta_{n+n_0}(-n-n_0) \geq 0.$$
Finally, we have 
$$\pi_n^{-1}(-n) = n_0 + \omega_{n+n_0}\left( \delta_{n+n_0}^{-1}(-n-n_0) \right) = n_0 + q_{n+n_0} \geq 0.$$
Thus, our elements $\pi_n$ satisfy the desired conditions.

\medskip \noindent
Next, fix an $m \geq 1$ and assume that $k:=\pi_m^{-1}(-m)<0$. Notice that, because $\pi_m$ fixes $(-\infty,-m)$ pointwise but does not fix $-m$, necessarily $\pi_m^{-1}(-m)>m$. Therefore, $1 \leq k <m$. For every $p \geq 0$, we define
$$\Omega(\pi_m,p):= \pi_m \circ \tau^{m-k}(\pi_m) \circ \tau^{2(m-k)}(\pi_m) \circ \cdots \circ \tau^{(2^p-1)(m-k)}(\pi_m).$$
Let us verify that the following assertions are satisfied:
\begin{itemize}
	\item[(i)] $\Omega(\pi_m,p+1) = \Omega(\pi_m,p) \circ \tau^{2^p(m-k)} ( \Omega(\pi_m,p))$ for every $p \geq 0$.
	\item[(ii)] $\Omega(\pi_m,p)^{-1}(-m) \geq 2^p-m$ for every $p \geq 0$.
	\item[(iii)] $\tau^s(\Omega(\pi_m,p)) \in Q$ for all $p \geq 0$ and $s \geq pn_0$.
	\item[(iv)] $\Omega(\pi_m,p)$ fixes $(- \infty,-m)$ pointwise.
	\item[(v)] $\Omega(\pi_m,p)(-m) = \pi_m(-m)$ for every $p \geq 0$.
\end{itemize}
Notice that $\Omega(\pi_m,p) = \Pi(\pi_m^{-1},p)^{-1}$ for every $p \geq 0$, where $\Pi(\cdot,\cdot)$ is the expression used in Claim~\ref{claim:Delta}. As a consequence, the items $(i)$, $(ii)$, and $(iv)$ follow from the corresponding assertions proved in Claim~\ref{claim:Delta}. The item $(iii)$ does not follow formally from the corresponding assertion proved in Claim~\ref{claim:Delta}, but the proof is exactly the same. Only our assertion $(v)$ requires a justification. We prove the equality by induction over $p$ thanks to $(i)$. For $p=0$, the equality is clear. For $p \geq 1$, we can write
$$\begin{array}{lcl} \Omega(\pi_m,p)(-m) & = & \Omega(\pi_m,p-1) \circ \tau^{2^{p-1}(m-k)} ( \Omega(\pi_m,p-1)) (-m) \\ \\ & = & \Omega(\pi_m,p-1) \left( 2^{p-1}(m-k) + \Omega(\pi_m, p-1) (-m-2^{p-1}(m-k) ) \right) \\ \\ & = & \Omega(\pi_m,p-1) ( -m) = \pi_m(-m)\end{array}$$
where the second equality is justified by $(iv)$. 

\medskip \noindent
We are now ready to construct our elements. Fix an $n \geq 1$ and let $m$ denote the smallest integer satisfying $m-n_0 \lceil \log_2(m) \rceil \geq n$. We define
$$\epsilon_n:= \left\{ \begin{array}{cl} \tau^{m-n} (\pi_m) & \text{if } \pi_m^{-1}(-m) \in R_2 \\ \tau^{m-n} \left( \Omega(\pi_m,\lceil \log_2(m) \rceil) \right) & \text{if } \pi_m^{-1}(-m) \notin R_2 \end{array} \right..$$
We can conclude as in Claim~\ref{claim:Delta}. Indeed, notice that $\epsilon_n \in Q$ as a consequence of $(iii)$ and because $\pi_m \in Q$; that $\epsilon_n$ fixes $(- \infty,-n)$ pointwise as a consequence of $(iv)$ and because $\pi_m$ fixes $(- \infty,-m)$ pointwise; that $\epsilon_n(-n)$ belongs to $R_2$ as a consequence $(v)$ and because $\pi_m(-m) \in R_2$; and finally that $\epsilon^{-1}_n(-n)$ belongs to $R_2$ as a consequence of $(ii)$. This concludes the proof of Claim~\ref{claim:Epsilon}. 

\medskip \noindent
For our next claim, we need to introduce some terminology. Given a set $X$, a point $x \in X$, and a permutation $\sigma \in \mathrm{Sym}(X)$, we denote by $\mathrm{ord}(\sigma,x)$ the smallest integer $k \geq 1$ such that $\sigma^k(x)=x$. Equivalently, $\mathrm{ord}(\sigma,x)$ is one more than the length of the cycle containing $x$ in the decomposition of $\sigma$ in a product of cycles with pairwise disjoint supports. 

\begin{claim}\label{claim:Eta}
For every $n \geq 1$, there exists $\eta_n \in Q$ such that:
\begin{itemize}
	\item $\eta_n$ fixes $(-\infty,-n)$ pointwise;
	\item $\eta_n(-n), \eta_n^{-1}(-n) \in R_2$;
	\item and $\mathrm{ord}(\eta_n,-n)=1+2^{\kappa(n)}$ where $\kappa(n):= 1+\lceil \log_2(n+n_0) \rceil$.
\end{itemize}
\end{claim}

\noindent
Fix an $n \geq 1$. We want to construct a permutation $\lambda_n$ satisfying the following conditions:
\begin{itemize}
	\item $\tau^{n_0}(\lambda_n) \in Q$;
	\item $\lambda_n$ fixes $(- \infty,-n)$ pointwise;
	\item $\lambda_n(-n), \lambda_n^{-1}(-n) \in R_2$;
	\item and $\mathrm{ord}(\lambda_n,-n) = 1+ 2^{\kappa(n-n_0)}$
\end{itemize}
Once such an element constructed for every $n \geq 1$, it will suffice to define $\eta_n:= \tau^{n_0}(\lambda_{n+n_0})$ for every $n \geq 1$ in order to conclude the proof of our claim. From now on, we focus on the construction of $\lambda_n$. 

\medskip \noindent
Let $\epsilon_n$ denote the permutation given by Claim~\ref{claim:Epsilon}. We can decompose it as
$$\epsilon_n = \left( x_1:=-n, x_2, \ldots, x_N \right) \circ \epsilon_n'$$
where $(x_1, \ldots, x_N)$ is the cycle containing $-n$ in the decomposition of $\epsilon_n$ as a product of pairwise disjoint cycles and where $\epsilon_n'$ is the product of the remaining cycles. Notice that $N= \mathrm{ord}(\epsilon_n,-n)$. Define a permutation $\theta$ as the product of the cycles of the following form. If $1 < i <j \leq N$ are such that $x_i, \ldots, x_j \notin R_1$ but both $x_{i-1}$ and $x_{j+1}$ (when $j < N$) belong to $R_1$, then $(x_{j-1}, x_{j-2}, \ldots, x_i)$ is a cycle of $\theta$. The picture to keep in mind is essentially the following, where white (resp.\ black) dots belong (resp.\ do not belong) to $R_1$. 

\begin{center}
\includegraphics[width=0.9\linewidth]{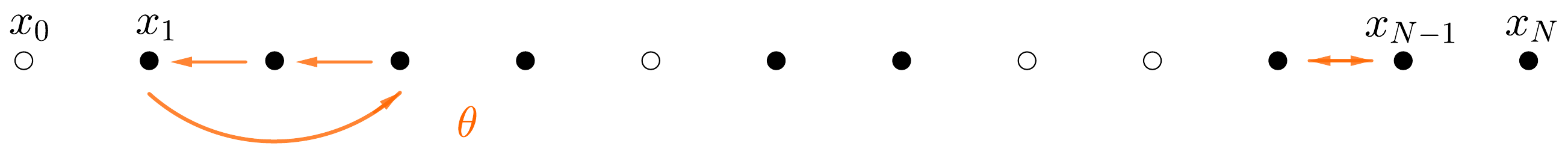}
\end{center}

\noindent
Notice that 
$$\epsilon_n \circ \theta = (x_1, \ldots, x_N) \circ \epsilon_n' \circ \theta = (x_1, \ldots, x_N) \circ \theta \circ \epsilon_n' = (x_{i_1}, \ldots, x_{i_r}) \circ \epsilon_n'$$
where $i_1< \cdots <i_r$ are such that each $x_{i_s}$ either belongs to $R_1$ or does not lie between two elements not in $R_1$ (i.e.\ one of $x_{i_s-1}$ and $x_{i_s+1}$ must belong to $R_1$). Notice that $x_{i_1}=-n$, $x_{i_2}=x_2 = \epsilon_n(-n)$, and that $x_{i_r} = \epsilon_n^{-1}(-n)$. Moreover, since at least a third of the $x_{i_s}$ must belong to $(-n,0)$, we have $\mathrm{ord}(\epsilon_n \circ \theta, -n) = r \leq 3n$. Consequently, if we set a sufficiently large integer $M \geq 1$ such that $\epsilon_n\circ \theta$ fixes $(M,+ \infty)$ pointwise, then the constant
$$M':= M+ 2^{\kappa(n-n_0)}+1 - \mathrm{ord}(\epsilon_n \circ \theta, -n)$$
is positive. Let us verify that
$$\lambda_n: = \epsilon_n \circ \theta \circ \left( \epsilon_n(-n), M+1, M+2, \ldots, M' \right)$$
satisfies the conditions we are interested in. Clearly, $\lambda_n \in Q^2$ so $\tau^{n_0}(\lambda_n) \in \tau^{n_0}(Q^2) \subset Q$. It is also clear by construction that $\lambda_n$ fixes $(-\infty,-n)$ pointwise. Then, by noticing that
$$\begin{array}{lcl} \lambda_n & = & \epsilon_n \circ \theta \circ (\epsilon_n(-n), M+1, M+2, \ldots, M') \\ \\ & = & (x_{i_1}, \ldots, x_{i_r}) \circ \epsilon_n' \circ (\epsilon_n(-n), M+1, M+2, \ldots, M') \\ \\ &= & (x_{i_1}, \ldots, x_{i_r}) \circ (x_{i_2}, M+1, M+2, \ldots, M') \circ \epsilon_n' \\ \\ & = & (x_{i_1}, x_{i_2}, M+1, M+2,\ldots, M', x_{i_3}, \ldots, x_{i_r}) \circ \epsilon_n', \end{array}$$
we can write
$$\lambda_n(-n) =\lambda_n(x_{i_1})=x_{i_2}= \epsilon_n(-n) \text{ and } \lambda_n^{-1}(-n) = \lambda_n^{-1}(x_{i_1})= x_{i_r} = \epsilon_n^{-1}(-n),$$
hence $\lambda_n(-n),\lambda_n^{-1}(-n) \in R_2$. Finally, notice that
$$\mathrm{ord}(\lambda_n) = r + M'- M = \mathrm{ord}(\epsilon_n \circ \theta,-n) + M'-M = 1+ 2^{\kappa(n-n_0)}.$$
This concludes the proof of Claim~\ref{claim:Eta}. 

\begin{claim}\label{claim:AlphaBeta}
For every $n \geq 1$, there exist $\alpha_n,\beta_n \in Q$ such that
\begin{itemize}
	\item $\alpha_n \circ \beta_n(-n)= \beta_n \circ \alpha_n(-n)=-n$;
	\item $\alpha_n$ and $\beta_n$ fix $(-\infty,-n)$ pointwise;
	\item and $\alpha(-n),\alpha^{-1}(-n) \in R_2$.
\end{itemize}
\end{claim}

\noindent
Fix an $n \geq 1$ and let $m$ denote the smallest integer satisfying $m-\kappa(m)n_0 \geq n$. Set
$$\alpha_n:= \tau^{m-n}\left( \eta_m \right)  \text{ and } \beta_n:= \tau^{m-n} \left( \eta_m^{2^{\kappa(m)}} \right).$$
Notice that $\alpha_n \circ \beta_n = \beta_n \circ \alpha_n$, and that
$$\alpha_n \circ \beta_n(-n) = \tau^{m-n} \left( \eta_m^{1+2^{\kappa(m)}} \right) (-n)= m-n + \eta_m^{1+2^{\kappa(m)}}(-m) = -n$$
since $\mathrm{ord}(\eta_m,-m)=1+2^{\kappa(m)}$. This proves the first item of our claim. In order to prove the second item, notice that, for all $p \geq 1$ and $j<-n$, we have
$$\tau^{m-n}(\eta_m^p)(j) = m-n + \eta_m^p(j-m+n) = m-n+j-m+n = j$$
since $\eta_m$ fixes $(-\infty,-m)$ pointwise. Finally, notice that
$$\alpha(-n) = m-n + \eta_m(-m) \geq 0 \text{ and } \alpha^{-1}(-n) = m-n + \eta_m^{-1}(-m) \geq 0.$$
This concludes the proof of Claim~\ref{claim:AlphaBeta}. 

\begin{claim}\label{claim:MuNu}
For every $n \geq 1$, there exist $\mu_n,\nu_n \in Q$ such that
\begin{itemize}
	\item $\mu_n(-n) =-n+1$ and $\mu_n^{-1}(-n) \geq 0$;
	\item $\mu_n \circ \nu_n(-n)= \nu_n \circ \mu_n(-n)=-n$;
	\item and $\mu_n,\nu_n$ fix $(- \infty,-n)$ pointwise.
\end{itemize}
\end{claim}

\noindent
Using the elements given by Claim~\ref{claim:AlphaBeta}, we define
$$\mu_{n-2n_0}:= \tau^{2n_0} \left( \tau(\beta_n) \circ \left( \alpha_n(-n), \alpha_n(-n)+1 \right) \circ \alpha_n \right)$$
and 
$$\nu_{n-2n_0}:= \tau^{2n_0} \left( \beta_n \circ (\alpha_n(-n), \alpha_n(-n) +1) \circ \tau(\alpha_n) \right)$$
for every $n>2n_0$. Notice that $\mu_{n-2n_0}$ and $\nu_{n-2n_0}$ both belong to $Q$ since $\alpha_n,\beta_n \in Q$ and according to Fact~\ref{fact:Bounded}. Then, notice that
$$\begin{array}{lcl} \mu_{n-2n_0}^{-1}(-n+2n_0) & = & 2n_0 + \alpha_n^{-1} \circ \left( \alpha_n(-n), \alpha_n(-n)+1 \right) \circ \tau(\beta_n^{-1})(-n) \\ \\ &= & 2n_0 + \alpha_n^{-1} \circ \left( \alpha_n(-n), \alpha_n(-n)+1 \right) \left( 1+ \beta_n^{-1} \right)(-n-1)  \\ \\ & = & 2n_0 + \alpha_n^{-1} \circ \left( \alpha_n(-n), \alpha_n(-n)+1 \right) (-n) \\ \\ & = & 2n_0 + \alpha_n^{-1}(-n) >0 \end{array}$$
where the last equality is justified by the fact that, since $\alpha_n(-n) \geq 0$, $-n$ cannot be equal to $\alpha_n(-n)$ nor to $\alpha_n(-n)+1$. We also have
$$\begin{array}{lcl} \mu_{n-2n_0}(-n+2n_0) & = & 2n_0 + \tau(\beta_n) \circ \left( \alpha_n(-n), \alpha_n(-n) +1 \right) \circ \alpha_n(-n) \\ \\ & =& 2n_0+1+ \beta_n \circ \alpha_n(-n) = -n+2n_0+1 \end{array}$$
This proves the first item of our claim. Similarly, we can compute
$$\begin{array}{lcl} \nu_{n-2n_0} (-n+2n_0) & = & 2n_0 + \beta_n \circ \left( \alpha_n(-n), \alpha_n(-n)+1 \right) \circ \tau(\alpha_n) (-n) \\ \\ & =& 2n_0 + \beta_n \circ \left( \alpha_n(-n), \alpha_n(-n)+1 \right) (1+ \alpha_n(-n-1)) \\ \\ & = & 2n_0 + \beta_n \circ \left( \alpha_n(-n), \alpha_n(-n)+1 \right)(-n) = 2n_0 + \beta_n(-n) \end{array}$$
where the last equality is again justified by that fact, since $\alpha_n(-n) \geq 0$, $-n$ cannot be equal to $\alpha_n(-n)$ nor to $\alpha_n(-n)+1$. The two latter computations allow us to write
$$\begin{array}{lcl} \mu_{n-2n_0} \circ \nu_{n-2n_0} (-n+2n_0) & =& \mu_{n-2n_0} \left( 2n_0 + \beta_n(-n) \right) \\ \\ & = & 2n_0 + \tau(\beta_n) \circ \left( \alpha_n(-n), \alpha_n(-n)+1 \right) \circ \alpha_n \circ \beta_n(-n) \\ \\ & =& 2n_0+1+ \beta_n(-n-1) = -n+2n_0 \end{array}$$
and
$$\begin{array}{lcl} \nu_{n-2n_0} \circ \mu_{n-2n_0} (-n+2n_0) & = & \nu_{n-2n_0} \left( -n+2n_0+1 \right) \\ \\ & = & 2n_0 + \beta_n \circ \left( \alpha_n(-n), \alpha_n(-n)+1 \right) \circ \tau(\alpha_n) (-n+1) \\ \\ & = & 2n_0 + \beta_n \circ \left( \alpha_n(-n), \alpha_n(-n)+1 \right) (\alpha_n(-n)+1) \\ \\ & = & 2n_0 + \beta_n \circ \alpha_n(-n)=-n+2n_0 \end{array}$$
This proves the second item of our claim. Finally, using the fact that $\alpha_n$ and $\beta_n$ fix $(-\infty,-n)$ pointwise, it is clear that
$$\begin{array}{lcl} \mu_{n-2n_0}(j) & = & 2n_0 + \tau(\beta_n) \circ \left( \alpha_n(-n), \alpha_n(-n)+1 \right) \circ \alpha_n(j-2n_0)  \\ \\ & = & 2n_0+ \tau(\beta_n) (j-2n_0) = 2n_0 +1 + \beta_n( j-2n_0-1) = j, \end{array}$$
and similarly $\nu_{n-2n_0}(j)=j$, for every $j<-n+2n_0$. This concludes the proof of Claim~\ref{claim:MuNu}. 

\begin{claim}\label{claim:Penultimate}
For every $n \geq 1$, there exists $\sigma_n \in Q$ such that
\begin{itemize}
	\item $\sigma_n$ switches $-n$ and $-n+1$;
	\item $\sigma_n$ is the identity on $(- \infty, -n)$.
\end{itemize}
\end{claim}

\noindent
Using the elements given by Claim~\ref{claim:MuNu}, we define
$$\sigma_{n-2n_0} := \tau^{2n_0} \left( \tau(\mu_n) \circ \left( \mu_n^{-1}(-n), \mu_n^{-1}(-n)+1 \right) \circ \nu_n \right)$$
for every $n>2n_0$. Notice that $\sigma_{n-2n_0}$ belongs to $Q$ since $\mu_n,\nu_n \in Q$ and according to Fact~\ref{fact:Bounded}. We have
$$\begin{array}{lcl} \sigma_{n-2n_0}(-n+2n_0) & =& 2n_0 + \tau(\mu_n) \circ \left( \mu_n^{-1}(-n), \mu_n^{-1}(-n)+1 \right) \circ \nu_n(-n) \\ \\ &= & 2n_0 + \tau(\mu_n) \circ \left( \mu_n^{-1}(-n), \mu_n^{-1}(-n)+1 \right) \circ \mu_n^{-1}(-n) \\ \\ &= & 2n_0 + \tau(\mu_n) \left( \mu_n^{-1}(-n) +1 \right)  = 2n_0 +1 + \mu_n(\mu_n^{-1}(-n))  \\ \\ & = & -n+2n_0+1 \end{array}$$
and
$$\begin{array}{lcl} \sigma_{n-2n_0}(-n+2n_0+1) & = & 2n_0 + \tau(\mu_n) \circ \left( \mu_n^{-1}(-n), \mu_n^{-1}(-n) +1 \right) \circ \nu_n(-n+1) \\ \\ & = & 2n_0 + \tau(\mu_n) \circ \left( \mu_n^{-1}(-n), \mu_n^{-1}(-n) +1 \right) \circ \nu_n( \mu_n(-n)) \\ \\ & = & 2n_0 + \tau(\mu_n) \circ \left( \mu_n^{-1}(-n), \mu_n^{-1}(-n) +1 \right) (-n) \\ \\ & =& 2n_0+1+ \mu_n(-n-1) = -n+2n_0  \end{array}$$
where the penultimate equality is justified by the fact that, since $\mu_n^{-1}(-n) \geq 0$, $-n$ cannot be equal to $\mu_n^{-1}(-n)$ nor to $\mu_n^{-1}(-n)+1$. This proves the first item of our claim. Next, notice that, for every $j<-n+2n_0$, we have
$$\begin{array}{lcl} \sigma_{n-2n_0}(j) & = & 2n_0 + \tau(\mu_n) \circ \left( \mu_n^{-1}(-n), \mu_n^{-1}(-n)+1 \right) \circ \nu_n(j-2n_0) \\ \\ & =& 2n_0 + \tau(\mu_n) (j-2n_0) = 2n_0+1+\mu_n(j-2n_0-1) = j\end{array}$$
This concludes the proof of Claim~\ref{claim:Penultimate}. 

\begin{claim}\label{claim:Final}
For all $n \geq 1$ and $i \geq -1$, there exists $\sigma_n^{(i)} \in Q$ such that
\begin{itemize}
	\item $\sigma^{(i)}_n$ switches $-n$ and $-n+1$;
	\item $\sigma_n^{(i)}$ is the identity on $(-\infty,-n) \cup [-n+2, -n+2+i]$.
\end{itemize}
\end{claim}

\noindent
We argue by induction over $i$. For $i=-1$, the existence of the $\sigma_n^{(i)}$ is given by Claim~\ref{claim:Penultimate}. From now on, assume that $i \geq 0$ and that $\sigma_n^{(i-1)}$ is already defined for every $n \geq 1$. 

\medskip \noindent
First of all, notice that we can assume that $\sigma_n^{(i-1)}(-n+2+i)$ belongs to $R_1 \cup R_2$ for every $n  \geq  1$. Indeed, for every $n \geq 1$, let $\pi_n$ denote an arbitrary transposition between $\sigma_n^{(i-1)}(-n+2+i)$ and a point of $R_2$ if $\sigma_n^{(i-1)}(-n+2+i) \notin R_2$. Otherwise, just set $\pi_n:= \mathrm{id}$. Then, 
$$\widetilde{\sigma}_n^{(i-1)} := \tau^{n_0} \left( \pi_{n+n_0} \circ \sigma_{n+n_0}^{(i-1)} \right)$$
is a permutation that belongs to $\tau^{n_0}(Q^2) \subset Q$, that switches $-n$ and $-n+1$, that is the identity on $(- \infty, -n) \cup [-n+2,-n+1+i]$, and that sends $-n+2+i$ in $R_1 \cup R_2$. From now on, we assume that $\sigma_n^{(i-1)}(-n+2+i)$ belongs to $R_1 \cup R_2$ for every $n  \geq  1$. 

\medskip \noindent
We fix an $n >n_0$ and we want to define an element $\sigma_{n-n_0}^{(i)}$ of $Q$ that switches $-n+n_0$ with $-n+n_0+1$ and that is the identity on $(-\infty,-n+n_0) \cup [-n+n_0+2, -n+n_0+2+i]$. 

\medskip \noindent
If $\sigma_n^{(i-1)}$ fixes $-n+2+i$, then we define
$$\sigma_{n-n_0}^{(i)}:= \tau^{n_0} \left( \sigma_n^{(i-1)} \right).$$
This is clearly an element of $Q$ that satisfies the conditions we are looking for. 

\medskip \noindent
From now on, assume that $\sigma_n^{(i-1)}$ does not fix $-n+2+i$. This implies that 
$$\sigma_n^{(i-1)}(-n+2+i) >-n+2+i.$$ 
For short, we set
$$k:= - \sigma_n^{(i-1)}(-n+2+i) \text{ and } a:= n-k-3-i.$$
The inequality above implies that $k<n-2-i$ and $a \geq 0$. We can define
$$\sigma_{n-n_0}^{(i)} := \tau^{n_0} \left( \tau^{an_0} \left( \sigma_{an_0+n-2-i}^{(i-1)} \circ \cdots \circ \sigma_{an_0+k+1}^{(i-1)} \right) \circ \sigma_n^{(i-1)} \right).$$
As a consequence of Fact~\ref{fact:Bounded}, this is an element of $Q$. Moreover, since the equality
$$\sigma_p^{(i-1)} \circ \cdots \circ \sigma_q^{(i-1)} (-q+1)=-p$$
holds for all $p \geq q$, we can write
$$\sigma_{an_0+n-2-i}^{(i-1)} \circ \cdots \circ \sigma_{an_0+k+1}^{(i-1)} (-an_0-k)= -an_0-n+2+i.$$
It follows that 
$$\tau^{an_0} \left( \sigma_{an_0+n-2-i}^{(i-1)} \circ \cdots \circ \sigma_{an_0+k+1}^{(i-1)} \right) (-k)= -n+2+i,$$
or equivalently
$$\tau^{an_0} \left( \sigma_{an_0+n-2-i}^{(i-1)} \circ \cdots \circ \sigma_{an_0+k+1}^{(i-1)} \right) \circ \sigma_n^{(i-1)} (-n+2+i)= -n+2+i.$$
Therefore, $\sigma_{n-n_0}^{(i)}$ fixes $-n+n_0+2+i$. Next, given a $j \in (-\infty,-2) \cup [0,i)$, we have
$$\begin{array}{lcl} \sigma_{n-n_0}^{(i)}(-n+n_0+2+j) & = & n_0 + \tau^{an_0} \left( \sigma_{an_0+n-2-i}^{(i-1)} \circ \cdots \circ \sigma_{an_0+k+1}^{(i-1)} \right) \circ \sigma_n^{(i-1)} (-n+2+j) \\ \\ &= & n_0 + \tau^{an_0} \left( \sigma_{an_0+n-2-i}^{(i-1)} \circ \cdots \circ \sigma_{an_0+k+1}^{(i-1)} \right)  (-n+2+j) \\ \\ & = & n_0+an_0 + \sigma_{an_0+n-2-i}^{(i-1)} \circ \cdots \circ \sigma_{an_0+k+1}^{(i-1)} (-an_0-n+2+j) \end{array}$$
Notice that, in this expression, the composition contains the $\sigma_{an_0+x}^{(i-1)}$ for $k+1 \leq x \leq n-2-i$. But $-an_0-n+2+j$ always belong to $(-\infty,-an_0-x)$ since $j < i \leq n-x-2$. Therefore, we have
$$\sigma_{n-n_0}^{(i)} (-n+n_0+2+j) = -n+n_0+2 + j.$$
So far, we have proved that $\sigma_{n-n_0}^{(i)}$ is the identity on $(-\infty,-n+n_0) \cup [-n+n_0+2, -n+n_0+2+i]$. It remains to verify that it switches $-n+n_0$ with $-n+n_0+1$. The computation is the same as above, namely:
$$\begin{array}{lcl} \sigma_{n-n_0}^{(i)}(-n+n_0) & = & n_0 + \tau^{an_0} \left( \sigma_{an_0+n-2-i}^{(i-1)} \circ \cdots \circ \sigma_{an_0+k+1}^{(i-1)} \right) \circ \sigma_n^{(i-1)} (-n) \\ \\ &= & n_0 + \tau^{an_0} \left( \sigma_{an_0+n-2-i}^{(i-1)} \circ \cdots \circ \sigma_{an_0+k+1}^{(i-1)} \right)  (-n+1) \\ \\ & = & n_0+an_0 + \sigma_{an_0+n-2-i}^{(i-1)} \circ \cdots \circ \sigma_{an_0+k+1}^{(i-1)} (-an_0-n+1) \\ \\ & = & -n+n_0 +1 \end{array}$$
where the last equality is justified by the fact that $-an_0-n+1$ is fixed by $\sigma_{an_0+x}^{(i-1)}$ for every $k+1 \leq x \leq n-2-i$; and we obtain $\sigma_{n-n_0}^{(i)}(-n+n_0+1) = -n+n_0$ similarly. This concludes the proof of Claim~\ref{claim:Final}. 

\medskip \noindent
We are now ready to conclude the proof of proposition. First, notice that, for every $n \geq 1$, we have
$$(-n, -n+1) = \tau^{n_0} (( -n-n_0, -n-n_0+1)) \in \tau \left( \sigma_{n+n_0}^{(n+n_0-2)} Q \right) \subset \tau^{n_0}(Q^2) \subset Q,$$
where $\sigma_{n+n_0}^{(n+n_0-2)} $ is given by Claim~\ref{claim:Final}. Thus, $Q$ contains all the transpositions between two consecutive vertices of $R_1$. As a consequence, if $\sigma$ is an arbitrary permutation supported on $R_1 \cup R_2$, we can write $\sigma = \sigma_1 \cdots \sigma_r$ where $\sigma_1, \ldots, \sigma_r$ are transpositions between consecutive vertices in $R_1 \cup R_2$, hence
$$\sigma = \tau^{pn_0} \left( \tau^{-pn_0}(\sigma_1) \circ \cdots \circ \tau^{-pn_0}(\sigma_r) \right) \in \tau^{pn_0} \left( Q^r \right) \subset Q$$
according to Fact~\ref{fact:Bounded}, where $p \geq r$ is chosen sufficiently large so that each $\tau^{-pn_0}(\sigma_i)$ is supported on $R_1$. Thus, every permutation supported on $R_1 \cup R_2$ belongs to $Q$.

\medskip \noindent
Finally, let $\xi \in \mathfrak{S}_\infty$ be an arbitrary permutation. Let $A \subset R_1$ be the set of the points of $R_1$ that are sent to $R_1^c$ by $\xi$, and $B \subset R_1^c$ the set of the points of $R_1^c$ that are sent to $R_1$ by $\xi$. Necessarily, $A$ and $B$ have the same cardinality, say $N$. Fix an enumeration $\{a_1, \ldots, a_N\}$ (resp.\ $\{b_1, \ldots, b_N \}$) of $A$ (resp.\ $B$), and $N$ points $x_1, \ldots, x_N$ (resp.\ $y_1, \ldots, y_N$) in $R_1$ (resp.\ $R_2$). Set
$$\mu := (a_1, x_1) \cdots (a_N,x_N) (b_1, y_1) \cdots (b_N, y_N)$$
and 
$$\nu: = (\xi(a_1),y_1) \cdots (\xi(a_N),y_N) (\xi(b_1), x_1) \cdots (\xi(b_N), x_N).$$
Notice that $\mu$ (resp.\ $\nu$) is the product of a permutation supported on $R_1$ with a permutation fixing $R_1$ pointwise. Now, $\nu \circ \xi \circ \mu$ exchanges $X$ and $Y$, stabilises $R_1 \backslash X$, and stabilises $R_1^c \backslash Y$. Consequently, $\nu \circ \xi \circ \mu$ can be written as the product of a permutation supported on $R_1 \cup R_2$ (which exchanges $X$ and $Y$), a permutation supported on $R_1$, and a permutation fixing $R_1$ pointwise. Hence
$$\begin{array}{lcl} \mathfrak{S}_\infty &  = & \left( \mathrm{Fix}(R_1) \mathrm{Sym}(R_1) \right) \cdot \left( \mathrm{Sym}(R_1 \cup R_2) \mathrm{Sym}(R_1) \mathrm{Fix}(R_1) \right) \cdot \left( \mathrm{Fix}(R_1) \mathrm{Sym}(R_1) \right) \\ \\ & = & \mathrm{Fix}(R_1) \mathrm{Sym}(R_1 \cup R_2) \mathrm{Fix}(R_1) \mathrm{Sym}(R_1) =\mathrm{Fix}(R_1) \mathrm{Sym}(R_1 \cup R_2) \mathrm{Fix}(R_1) \end{array}$$
We record this observation for future use:

\begin{fact}\label{fact:SymBounded}
For all distinct $1 \leq i,j \leq n$, 
$$\mathfrak{S}_\infty=\mathrm{Fix}(R_i) \mathrm{Sym}(R_i \cup R_j) \mathrm{Fix}(R_i).$$
\end{fact}

\noindent
Since we know that $\mathrm{Fix}(R_1)$ and $\mathrm{Sym}(R_1 \cup R_2)$ are both contained in $Q$, it follows that $\mathfrak{S}_\infty \subset Q^3$. Hence
$$\mathfrak{S}_\infty  = \tau^{3n_0}( \mathfrak{S}_\infty ) \subset \tau^{3n_0}(Q^3) \subset Q$$
according to Fact~\ref{fact:Bounded}. This concludes the proof of our proposition. 
\end{proof}

\begin{proof}[Proof of Theorem~\ref{thm:FocalTwo}.]
First, we know from Lemma~\ref{lem:SecondNotComparable} that the subsets provided by our theorem are indeed focal hyperbolic structures of $H_n(2)$, and that they are distinct. Next, let $\mathfrak{X}$ be a focal hyperbolic structure on $H_n(2)$. According to Corollary~\ref{cor:FocalDes}, there exists $Q \subset \mathfrak{S}_\infty$ such that $\tau$ or $\tau^{-1}$ is strictly confining $\mathfrak{S}_\infty$ and such that $\mathfrak{X} = [Q \cup \{t\}]$ or $[Q \cup \{ t^{-1} \}]$. If $\tau$ is strictly confining $\mathfrak{S}_\infty$, then it follows from Proposition~\ref{prop:StrictCofining} that $\mathfrak{X}= [\mathrm{Fix}(R_1) \cup \{t\}]$. Otherwise, if $\tau^{-1}$ is strictly confining $\mathfrak{S}_\infty$, then, by applying the automorphism of $H_n(2)$ induced by the automorphism of $\mathfrak{R}_n$ that switches $R_1$ and $R_2$ but fixes all the other rays pointwise, we deduce symmetrically from Proposition~\ref{prop:StrictCofining} that $\mathfrak{X}=[\mathrm{Fix}(R_2) \cup \{t^{-1}\}]$. 
\end{proof}

\subsection{Hyperbolic structures}\label{section:PartialHyp}

\noindent
Thanks to Theorem~\ref{thm:FocalTwo}, we are now ready to conclude the proof of Theorem~\ref{thm:SecondPartialHoughton}. 

\begin{proof}[Proof of Theorem~\ref{thm:SecondPartialHoughton}.]
We know from Theorem~\ref{thm:FocalTwo} that $H_n(2)$ has exactly two focal hyperbolic structures, namely $[\mathrm{Fix}(R_1) \cup \{t_{1,2}\}]$ and $[\mathrm{Fix}(R_2) \cup \{t_{1,2}^{-1}\}]$. They are not comparable according to Lemma~\ref{lem:SecondNotComparable}. As a consequence of Lemma~\ref{lem:NonOrientedLineal} and Fact~\ref{fact:NotOntoDinfty} below, the lineal hyperbolic structures of $H_n(2)$ are all oriented. Since $H_n(2)$ decomposes as a semidirect product $\mathfrak{S}_\infty \rtimes \mathbb{Z}$, it is clear that $H^1(H_n(2),\mathbb{R})$ is generated by the quotient map $H_n(2) \twoheadrightarrow H_n(2)/ \mathfrak{S}_\infty$. We deduce from Lemma~\ref{lem:LinealActions} that $H_n(2)$ has a single lineal action, namely $[\mathfrak{S}_\infty \cup \{t_{1,2}\}]$. The fact that the two focal structures dominate the lineal structure is clear.

\begin{fact}\label{fact:NotOntoDinfty}
For every $n \geq 2$, $H_n(2)$ does not surject onto $\mathbb{D}_\infty$.
\end{fact}

\noindent
On the one hand, in $\mathbb{D}_\infty$, any two distinct order-$2$ elements generate an infinite dihedral group. On the other hand, the subgroup $\mathfrak{S}_\infty$ is generated by transpositions and any two distinct transpositions generate a finite dihedral group. Therefore, the image of every morphism $\mathfrak{S}_\infty \to \mathbb{D}_\infty$ is either trivial or cyclic of order two. Since the centraliser of an order-$2$ element in $\mathbb{D}_\infty$ is finite, the image of every morphism $H_n(2) \to \mathbb{D}_\infty$ must be either finite or infinite cyclic. In any case, the morphism cannot be surjective. 
\end{proof}

\section{Higher rank Houghton groups}

\noindent
Hyperbolic structures on the second Houghton group are understood thanks to Theorem~\ref{thm:SecondPartialHoughton}. In this section, we focus on Houghton groups $H_n$ for $n \geq 3$. 

\begin{thm}\label{thm:HoughtonHigher}
Let $n\geq 3$ be an integer. The Houghton group $H_n$ has exactly $n$ focal hyperbolic structures, namely
$$\mathcal{F}_i:= [\mathrm{Fix}(R_i) \cup T] \text{ for } 1 \leq i \leq n, \text{ where } T:= \{ t_{r,s}, 1 \leq r,s \leq n\};$$
which are pairwise non-comparable. Each $\mathcal{F}_i$ dominates a single lineal hyperbolic structure, namely $[ \mathrm{ker}(\lambda_i) \cup T].$ There are uncountably many lineal hyperbolic structures, which are all oriented and pairwise non-comparable.
\end{thm}

\noindent
Recall from Section~\ref{section:Preliminaries} that, for $1 \leq i \leq n$, $\lambda_i$ denotes the morphism $H_n \twoheadrightarrow \mathbb{Z}$ that records the eventual algebraic translation length of an element along the ray $R_i$.

\subsection{Possible focal actions}

\noindent
Towards the proof of Theorem~\ref{thm:HoughtonHigher}, we start by describing the possible focal hyperbolic structures on our Houghton group.

\begin{prop}\label{prop:FocalGeneral}
For every $n \geq 3$, $H_n$ has exactly $n$ focal hyperbolic structures, namely
$$ \left[ \mathrm{Fix}(R_i) \cup T \right] \text{ for } 1 \leq i \leq n, \text{ where } T:= \{ t_{r,s}, 1 \leq r,s \leq n\}.$$ 
\end{prop}

\noindent
First of all, we verify that the subsets given by our proposition are indeed focal hyperbolic structures.

\begin{lemma}\label{lem:TheyAreHypS}
For all $n \geq 3$ and $1 \leq i \leq n$, $\mathrm{Fix}(R_i) \cup T $ defines a focal hyperbolic structure on $H_n$. Moreover, these structures are pairwise non-comparable.
\end{lemma}

\begin{proof}
Given an index $1 \leq i \leq n$, fix an arbitrary $1 \leq j \leq n$ distinct from $i$. The morphism $\lambda_i : H_n \twoheadrightarrow \mathbb{Z}$ allows us to decompose $H_n$ as $\mathrm{ker}(\lambda_i) \rtimes \langle t_{i,j} \rangle$. Let us verify that $t_{i,j}$ is strictly confining $\mathrm{ker}(\lambda_i)$ in $\mathrm{Fix}(R_i)$. 

\medskip \noindent
First, notice that 
$$t_{i,j} \mathrm{Fix}(R_i) t_{i,j}^{-1} = \mathrm{Fix} ( \{i\} \times [1, + \infty)) \subsetneq \mathrm{Fix}(R_i)$$
and that $t_{i,j}^0 \mathrm{Fix}(R_i)^2 t_{i,j}^{-0} = \mathrm{Fix}(R_i)$. Finally, for every $g \in \mathrm{ker}(\lambda_i)$, we can fix a $p \geq 0$ such that $g$ fixes $\{i\} \times [p,+ \infty)$ pointwise and then we have $t_{i,j}^p g t_{i,j}^{-p} \in \mathrm{Fix}(R_i)$. Therefore, $\mathrm{ker}(\lambda_i) = \bigcup_{m \geq 0} t_{i,j}^{-m} \mathrm{Fix}(R_i) t_{i,j}^m = \mathrm{ker}(\lambda_i)$. 

\medskip \noindent
We deduce from Proposition~\ref{prop:ConfiningToHyp} that $[\mathrm{Fix}(R_i) \cup \{t_{i,j}\}]$, and a fortiori $[\mathrm{Fix}(R_i) \cup T]$, is a focal hyperbolic structure of $H_n$. Moreover, we also know from Proposition~\ref{prop:ConfiningToHyp} that the Busemann morphism takes a non-zero value on $t_{i,j}$, which implies that $\|t_{i,j}^m\|_{\mathrm{Fix}(R_i) \cup T} \to + \infty$ as $m \to + \infty$. 

\medskip \noindent
Now, if $1 \leq k \leq n$ is an index distinct from $i$, we need to justify that the two focal hyperbolic structures $[\mathrm{Fix}(R_i) \cup T]$ and $[\mathrm{Fix}(R_k) \cup T]$ are distinct. Because $n \geq 3$, and since $j$ was an arbitrary index distinct from $i$, we can assume that $j$ is also distinct from $k$. Therefore, $t_{i,j}$ belongs to $\mathrm{Fix}(R_k)$. Since $(t_{i,j}^m)_{m \geq 0}$ remains bounded in $(H_n, \mathrm{Fix}(R_k) \cup T)$ but diverges in $(H_n, \mathrm{Fix}(R_i) \cup T)$, we conclude that $[\mathrm{Fix}(R_i) \cup T] \preceq [\mathrm{Fix}(R_k) \cup T]$ does not hold. 

\medskip \noindent
Thus, we have proved that our focal hyperbolic structures are pairwise non-comparable.
\end{proof}

\noindent
The key observation, in order to prove Proposition~\ref{prop:FocalGeneral}, is that second partial Houghton subgroups are necessarily cobounded with respect to a focal hyperbolic structure, which will allow us to apply Theorem~\ref{thm:SecondPartialHoughton}. This observation is a consequence of the following result, which can be found in \cite[Proposition~4.5]{MR3420526}.

\begin{lemma}\label{lemma:Cobounded}
Let $G$ be a group and $[X]$ a regular focal hyperbolic structure. Let $\beta : G \to \mathbb{Z}$ denote the corresponding Busemann morphism. For every $g \in G$ satisfying $\beta(g) \neq 0$, the subgroup $\langle [G,G], g \rangle = [G,G] \rtimes \langle g \rangle$ is cobounded in $(G,d_X)$.
\end{lemma}

\begin{proof}[Proof of Proposition~\ref{prop:FocalGeneral}.]
Let $[X]$ be a focal hyperbolic structure of $H_n$. Let $\beta : H_n \to \mathbb{Z}$ denote the corresponding Busemann morphism. Since the (images of) $t_{i,j}$ generate the abelianisation of $H_n$, necessarily there exist $i \neq j$ such that $\beta(t_{i,j}) \neq 0$. Up to replacing $t_{i,j}$ with its inverse $t_{j,i}$, we can assume without loss of generality that $\beta(t_{i,j})>0$. 

\medskip \noindent
As a consequence of Lemma~\ref{lemma:Cobounded}, the partial Houghton subgroup $H:= H_n(\{i,j\})$, which is generated by $\mathfrak{S}_\infty= [H_n,H_n]$ and $t_{i,j}$, is cobounded in $(H_n, d_X)$. Therefore, there exists a generating set $Y \subset H$ such that the inclusion map $H \hookrightarrow H_n$ induces a quasi-isometry $(H,d_Y) \to (H_n,d_X)$ (see for instance \cite[Lemma~3.11]{MR3995018}). Necessarily, $[Y]$ yields a focal hyperbolic structure for $H$. But we know from Theorem~\ref{thm:FocalTwo} and Proposition~\ref{prop:ConfiningToHyp} that $H$ admits a single focal hyperbolic structure for which $t_{i,j}$ has a positive image under the corresponding Busemann morphism, namely $[\mathrm{Fix}_\infty(R_i) \cup \{t_{i,j}\}]$, where we set $\mathrm{Fix}_\infty(R_i):= \mathrm{Fix}(R_i) \cap \mathfrak{S}_\infty$ for short. From now on, we will assume for convenience that $Y= \mathrm{Fix}_\infty(R_i) \cup \{t_{i,j}\}$.

\medskip \noindent
Setting $T_i:= \langle t_{r,s}, \ r,s \neq i \rangle$, our goal is to show that $[X]=[Y \cup T_i]$. We start by proving two claims.

\begin{claim}\label{claim:Tbounded}
The subgroup $T_i$ is bounded in $(H_n,d_X)$.
\end{claim}

\noindent
It amounts to showing that $T_i$ has bounded orbits with respect to its action on the Cayley graph $\mathrm{Cayl}(H_n,X)$. Notice that we already know that $[T_i,T_i] \subset \mathrm{Fix}(R_i) \subset X$ has bounded orbits. According to Corollary~\ref{cor:BoundedOrbits}, it suffices to show that each $t_{r,s}$, for $r,s \neq i$, has bounded orbits as well. Assuming to the contrary that there exists some $r,s \neq i$ such that $t_{r,s}$ is loxodromic, then we can deduce as above that $\mathrm{Fix}(R_r)$ or $\mathrm{Fix}(R_s)$ has bounded orbits. More precisely, assuming that $\beta(t_{r,s})>0$ up to replacing $t_{r,s}$ with $t_{s,r}$, it follows from Lemma~\ref{lemma:Cobounded} that the partial Houghton subgroup $K:=H_n(\{r,s\})$ is cobounded in $(H_n,d_X)$. Therefore, there exists a generating set $Z \subset K$ such that inclusion $K \hookrightarrow H_n$ induces a quasi-isometry $(K,d_Z) \to (H_n,d_X)$. Necessarily, $[Z]$ yields a focal hyperbolic structure for $K$. But we know from Theorem~\ref{thm:FocalTwo} and Proposition~\ref{prop:ConfiningToHyp} that $K$ admits a single focal hyperbolic structure for which $t_{r,s}$ has a positive image under $\beta$, namely $[\mathrm{Fix}(R_r) \cup \{t_{r,s}\}]$. Therefore, $\mathrm{Fix}(R_r)$ must be bounded in $(K,d_Z)$, and a fortiori in $(H_n,d_X)$. Thus, we know that $\mathrm{Fix}(R_i)$ and $\mathrm{Fix}(R_r)$ both have bounded orbits in $\mathrm{Cayl}(H_n,X)$. Because $n \geq 3$, we deduce from Lemma~\ref{lem:BoundedOrbits} and Fact~\ref{fact:SymBounded} that $\mathfrak{S}_\infty$ entirely has bounded orbits. According to Lemma~\ref{lem:CommSubEll}, this contradicts the fact that $[X]$ is a focal hyperbolic structure. This concludes the proof of Claim~\ref{claim:Tbounded}. 

\begin{claim}\label{claim:Lip}
For all $h_1,h_2 \in H$,
$$d_{Y \cup T_i}(h_1,h_2) \leq d_Y(h_1,h_2) \leq 4d_{Y \cup T_i}(h_1,h_2) .$$
\end{claim}

\noindent
The inequality $d_{Y\cup T_i}(h_1,h_2) \leq d_Y(h_1,h_2)$ is clear, so we focus on the second inequality. Recall that, for every $1 \leq k \leq n$, $\lambda_k : H_n \twoheadrightarrow \mathbb{Z}$ denotes the morphism that gives the eventual algebraic translation length of an element along the ray $R_k$. For every $h \in H_n$, set 
$$\lambda(h):= \prod\limits_{k \neq i,j} t_{k,j}^{\lambda_k(h)}.$$
Notice that $h\lambda(h) \in H$ for every $h \in H_n$. Moreover, for all $h \in H_n$ and $s \in Y \cup T_i$, $d_Y(h\lambda(h),hs\lambda(hs)) \leq 4$. Indeed:
\begin{itemize}
	\item if $s \in \mathrm{Fix}_\infty(R_i)$, then $\lambda(hs)=\lambda(h)$ and $$d_Y(h \lambda(h), hs\lambda(hs)) = \| \lambda(h)^{-1} s \lambda(h) \|_Y \leq 1$$ since $\lambda(h)^{-1}s \lambda(h)$ is a permutation that fixes $R_i$;
	\item if $s \in T_i$, then, noticing that $\lambda_k( \lambda(h)^{-1}s \lambda(hs)) = - \lambda_k(h) + \lambda_k(s) + \lambda_k(hs) = 0$ for every $k \neq i,j$, we deduce that $$d_Y(h \lambda(h), hs\lambda(hs)) = \| \lambda(h)^{-1} s \lambda(h) \|_Y \leq 1$$ since $\lambda(h)^{-1}s \lambda(hs)$ is a permutation that fixes $R_i$;
	\item if $s = t_{i,j}$, then notice that $\lambda(hs)=\lambda(h)$, and that $\lambda(h)^{-1}t_{i,j} \lambda(h) t_{i,j}^{-1}$ is a permutation that fixes $R_i$, which implies $\|\lambda(h)^{-1}t_{i,j} \lambda(h) t_{i,j}^{-1}\|_Y \leq 3$ according to Fact~\ref{fact:Norm}; from which we deduce that
$$d_Y(h \lambda(h), hs\lambda(hs)) = \| \lambda(h)^{-1} s \lambda(h) \|_Y \leq 4.$$
\end{itemize}
Therefore, the map $(H_n,d_{Y \cup T_i})  \to (H, d_Y)$ defined by $h \mapsto h \lambda(h)$ yields a $4$-Lipschitz retraction, providing the desired conclusion. Indeed, given a geodesic $a_0, \ldots, a_n$ in $(H_n,d_{Y \cup T_i})$ from $h_1$ to $h_2$, we have
$$d_Y(h_1,h_2) \leq \sum\limits_{i=0}^{n-1} d_Y ( a_i \lambda(a_i), a_{i+1}\lambda(a_{i+1})) \leq 4 n = 4 d_{Y \cup T_i}(h_1,h_2).$$
This concludes the proof of Claim~\ref{claim:Lip}. 

\medskip \noindent
Now, let $g_1,g_2 \in H_n$ be two arbitrary elements. Let $D$ denote the diameter of $T_i$ in $(H,d_X)$, which is finite according to Claim~\ref{claim:Tbounded}; and let $K>0$ be a constant such that the inclusion map $(H,d_Y) \to (H_n, d_X)$ yields a $(K,K)$-quasi-isometry. We have
$$\begin{array}{lcl} d_X(g_1,g_2) &  \leq & d_X(g_1\lambda(g_1), g_2 \lambda(g_2)) +2D \leq K d_Y(g_1 \lambda(g_1), g_2 \lambda(g_2)) +K+2D \\ \\ & \leq & 4K d_{Y \cup T_i} (g_1 \lambda(g_1), g_2 \lambda(g_2)) +K+2D \leq 4K d_{Y \cup T_i} (g_1,g_2) +K+2(D+1)\end{array}$$
and
$$\begin{array}{lcl} d_X(g_1,g_2)& \geq & d_X(g_1 \lambda(g_1),g_2 \lambda(g_2)) - 2D \geq \frac{1}{K} d_Y(g_1 \lambda(g_1), g_2 \lambda(g_2)) -K-2D \\ \\ & \geq & \frac{1}{K} d_{Y \cup T_i}(g_1 \lambda(g_1), g_2 \lambda(g_2)) -K-2D \geq \frac{1}{K} d_{Y \cup T_i}(g_1, g_2 ) -K-2(D+1) \end{array}$$
This concludes the proof of the fact that $[X]=[Y \cup T_i]$. By adding finitely many elements to our hyperbolic structure, we find that
$$[X] = [ \mathrm{Fix}_\infty(R_i) \cup \langle t_{r,s}, \ r,s \neq i \rangle \cup \{t_{i,r}, \ r \neq i \} ].$$
In order to justify that $[X]= [\mathrm{Fix}(R_i) \cup T]$, it suffices to notice that:

\begin{claim}\label{claim:FinalClaim}
The equality $\mathrm{Fix}(R_i)= \mathrm{Fix}_\infty(R_i) \cdot \langle t_{r,s}, \ r,s \neq i \rangle$ holds. 
\end{claim}

\noindent
Let $g \in \mathrm{Fix}(R_i)$ be an element. Because $n \geq 3$, for every $j \neq i$ we can find some $a_j \in \langle t_{r,s}, \ r,s \neq i \rangle$ such that $\lambda_j(g)= \lambda_j(a_j)$. Then, the product of the inverses of the $a_j$, in an arbitrary order, yields an element $a \in \langle t_{r,s}, \ r,s \neq i\rangle$ such that $\lambda_j(ga)=0$ for every $j \neq i$. Since $g$ and $a$ both fix $R_i$ pointwise, it follows that $ga$ is a permutation that belongs to $\mathrm{Fix}_\infty(R_i)$. This concludes the proof of Claim~\ref{claim:FinalClaim}.

\medskip \noindent
Thus, we have proved that all the focal hyperbolic structures of $H_n$ are of the form given by our proposition. Conversely, Lemma~\ref{lem:TheyAreHypS} allows us to conclude. 
\end{proof}

\subsection{Hyperbolic structures}

\noindent
Thanks to Proposition~\ref{prop:FocalGeneral}, we are now ready to prove Theorem~\ref{thm:HoughtonHigher}.

\begin{proof}[Proof of Theorem~\ref{thm:HoughtonHigher}.]
Focal hyperbolic structures of $H_n$ are described by Proposition~\ref{prop:FocalGeneral}. Recall from Lemma~\ref{lem:abelianisation} that
$$\lambda_1 \oplus \cdots \oplus \lambda_{n-1} : H_n \to \mathbb{Z}^{n-1}$$
is the abelianisation map. Consequently, $\{\lambda_1, \ldots, \lambda_{n-1}\}$ yields a basis of $H^1(H_n, \mathbb{R})$. Since $n \geq 3$, it follows from Lemma~\ref{lem:LinealActions} that $H_n$ has at least two oriented lineal hyperbolic structures. But then \cite[Theorem~4.22]{MR3995018} implies that $H_n$ are uncountably many oriented lineal hyperbolic structures, which are pairwise non-comparable. Recall from Lemma~\ref{lem:NonOrientedLineal} and Fact~\ref{fact:NotOntoDinfty} that every lineal hyperbolic structure of $H_n$ is oriented.

\medskip \noindent
Since, for every $1 \leq i \leq n$, it is clear that the focal hyperbolic structure $[\mathrm{Fix}(R_i) \cup T]$ dominates the lineal hyperbolic structure $[\mathrm{ker}(\lambda_i) \cup T]$, and since a focal hyperbolic structure cannot dominates two distinct lineal hyperbolic structures, this completes our description of the poset of hyperbolic structures of $H_n$. 
\end{proof}

\section{Groups with few focal actions}

\noindent
As already used several times in the previous sections, the group $\mathrm{QAut}(\mathcal{R}_n)$ contains a natural copy of the symmetric group $\mathrm{Sym}(n)$ that normalises $H_n$. Namely, $\mathrm{Sym}(n)$ permutes the $n$ rays of $\mathcal{R}_n$. This allows us to define the following family of virtual Houghton groups:

\begin{definition}
Given an integer $n \geq 1$ and a subgroup $G \leq \mathrm{Sym}(n)$, the corresponding \emph{Houghton group with permutations} is $H_n(G):= H_n \rtimes G$. 
\end{definition}

\noindent
Despite coinciding with Houghton groups up to finite index, Houghton groups with permutations provide interesting examples of posets of hyperbolic structures. Most notably, they provide examples of groups admitting a single focal hyperbolic structure. 

\begin{thm}\label{thm:HoughtonPerm}
Let $n \geq 2$ be an integer and $G \leq \mathrm{Sym}(n)$ a subgroup. The only possible focal hyperbolic structures of $H_n(G)$ are
$$[\mathrm{Fix}(R_i) \cup T \cup G] \text{ for } i \in \{1, \ldots, n\} \text{ fixed by } G,$$
where $T:= \{ t_{r,s}, \ 1 \leq r,s \leq n\}$. They are pairwise non-comparable. Each $\mathcal{F}_i$ dominates a single lineal hyperbolic structure, namely $[\mathrm{ker}(\lambda_i) \cup T \cup G]$.
\end{thm}

\noindent
Formally, Theorem~\ref{thm:HoughtonPerm} will be a direct consequence of the description of hyperbolic structures for Houghton groups as given by Theorem~\ref{thm:Intro}. In order to do this, we need to understand how hyperbolic structures behave under finite extensions. 

\medskip \noindent
Let $G$ be an arbitrary group. It is clear that 
$$\varphi \ast [X] = [\varphi(X)], \ \varphi \in \mathrm{Aut}(G), [X] \in \mathcal{H}(G)$$
defines an action of $\mathrm{Aut}(G)$ on $\mathcal{H}(G)$ that preserves both $\leq$ and the labels (i.e.\ whether a hyperbolic structure is ellptic, lineal, focal, or general type). Given a subgroup $F \leq \mathrm{Aut}(G)$, we say that a hyperbolic structure $\mathscr{X} \in \mathcal{H}(G)$ is \emph{$F$-invariant} if $\varphi \ast \mathscr{X}=\mathscr{X}$ for every $\varphi \in F$; and we denote by $\mathcal{H}^F(G)$ the set of $F$-invariant hyperbolic structures of $G$. Then:

\begin{prop}\label{prop:HypStructureFI}
Let $G$ be a group and $F \leq \mathrm{Aut}(G)$ a finite subgroup. The map
$$\Psi : \left\{ \begin{array}{ccc} \mathcal{H}^F(G) & \to & \mathcal{H}(G \rtimes F) \\ \left[ X \right] & \mapsto & \left[ X \cup F \right] \end{array} \right.$$
induces a label-preserving poset-isomorphism. 
\end{prop}

\noindent
Before proving Proposition~\ref{prop:HypStructureFI}, let us state a few equivalent characterisations of invariant hyperbolic structures.

\begin{lemma}\label{lem:Finvariant}
Let $G$ be a group and $F \leq \mathrm{Aut}(G)$ a finite subgroup. Given a hyperbolic structure $\mathscr{X} \in \mathcal{H}(G)$, the following assertions are equivalent:
\begin{itemize}
	\item[(i)] $\mathscr{X}$ is $F$-invariant, i.e.\ $\varphi \ast \mathscr{X}=\mathscr{X}$ for every $\varphi \in F$;
	\item[(ii)] for all $X \in \mathscr{X}$ and $\varphi \in F$, $[X]=[\varphi(X)]$;
	\item[(iii)] for every $X \in \mathscr{X}$, $[X]= \left[ \bigcup_{\varphi \in F} \varphi(X) \right]$;
	\item[(iv)] there exists $X \in \mathscr{X}$ such that $\varphi(X)=X$ for every $\varphi \in F$.
\end{itemize}
\end{lemma}

\begin{proof}
The equivalence $(i) \Leftrightarrow (ii)$ just follows from the definition of the action of $\mathrm{Aut}(G)$ on $\mathcal{H}(G)$. The implication $(ii) \Rightarrow (iii)$ is a consequence of the following straightforward observation:

\begin{fact}\label{fact:UnionHyp}
For all $[A],[B] \in \mathcal{H}(G)$, if $[A]=[B]$ then $[A]= [A \cup B]$.
\end{fact}

\noindent
In order to prove $(iii) \Rightarrow (iv)$, fix an arbitrary representative $X_0 \in \mathscr{X}$ and set $X:= \bigcup_{\varphi \in F} \varphi(X)$. Then, $X \in \mathscr{X}$ since $[X]=[X_0]$; and, by construction, $\varphi(X)=X$ for every $\varphi \in F$. Finally, in order to prove $(iv) \Rightarrow (ii)$, fix a representative $X \in \mathscr{X}$ satisfying $\varphi(X)=X$ for every $\varphi  \in F$, and notice that, for all $Y \in \mathscr{X}$ and $\varphi \in F$, we have $[\varphi(Y)] =[\varphi(X)] = [X] = [Y]$. 
\end{proof}

\begin{proof}[Proof of Proposition~\ref{prop:HypStructureFI}.]
First of all, notice that, for all $[X], [Y] \in \mathcal{H}^F(G)$, if $[X] \leq [Y]$ then $[X \cup F] \leq [Y \cup F]$. Moreover, if $[X] \in \mathcal{H}^F(G)$ with $\varphi(X)=X$ for every $\varphi \in F$, then $\mathrm{Cayl}(G \rtimes F, X \cup F)$ decomposes as the product $\mathrm{Cayl}(G,X) \times K(F)$, where $K(F)$ denotes the complete graph on $F$. This shows that $\Psi$ is a well-defined, labelled-preserving, poset-morphism. It also follows that $\Psi$ is injective. Indeed, if $[X]$ and $[Y]$ are two $F$-invariant hyperbolic structures satisfying $[X \cup F]= [Y \cup F]$, then, according to Lemma~\ref{lem:Finvariant}, we can assume that $\varphi(X)=X$ and $\varphi(Y)=Y$ for every $\varphi \in F$, so we deduce from the fact that 
$$\mathrm{Cayl}(G \rtimes F, X \cup F)= \mathrm{Cayl}(G, X) \times K(F)$$
and 
$$\mathrm{Cayl}(G \rtimes F, Y \cup F) = \mathrm{Cayl}(G,Y) \times K(F)$$
are quasi-isometric that $\mathrm{Cayl}(G, X)$ and $\mathrm{Cayl}(G,Y)$ are quasi-isometric, which amounts to saying that $[X]=[Y]$. 

\medskip \noindent
It remains to verify that $\Psi$ is surjective. So let $[X_0] \in \mathcal{H}(G \rtimes F)$ be an arbitrary hyperbolic structure. Because $F$ is finite, we know that $[X_0]=[X_0 \cup F]$. Fixing an enumeration $X_0= \{ g_if_i, \ i \in I\}$ where $g_i \in G$ and $f_i \in F$ for every $i \in I$, set $X_1:= \{ g_i, \ i \in I\}$. Then 
$$[X_0 \cup F]= [X_1 \cup F] = \left[ \bigcup\limits_{f \in F} fX_1f^{-1} \cup F \right] = \left[ \bigcup\limits_{ \varphi \in F} \varphi(X_1) \cup F \right].$$
Set $X_2:= \bigcup_{\varphi \in F} \varphi(X_1)$. Then $[X_2] \in \mathcal{H}^F(G)$ and $\Psi([X_2]) = [X_0]$.  
\end{proof}

\begin{proof}[Proof of Theorem~\ref{thm:HoughtonPerm}.]
Our theorem is an immediate consequence of Theorem~\ref{thm:Intro} and Proposition~\ref{prop:HypStructureFI}. 
\end{proof}

\begin{proof}[Proof of Theorem~\ref{thm:IntroSingle}.]
Let $n > k + 1$ be two non-negative integers. Let $F:= \mathrm{Fix}(\{1, \ldots, k\})$ be the pointwise fixator of $\{1, \ldots, k\}$ in $\mathrm{Sym}(n)$. Because $n>k+1$, the points $1, \ldots, k$ are the only points of $\{1, \ldots, n\}$ that are globally fixed by $F$. The desired conclusion then follows from Theorem~\ref{thm:HoughtonPerm}.
\end{proof}

\begin{remark}
For all integers $n>k+1$, Theorem~\ref{thm:IntroSingle} yields a group of type $F_n$ with exactly $k$ focal hyperbolic structures. It is worth mentioning that, following the same idea, there is a good candidate for an example of type $F_\infty$. For an arbitrary set $S$, possibly infinite, define the Houghton group $H(S)$ as the group of the quasi-automorphisms of the disjoint union $\mathcal{R}(S)$ of rays $S \times \mathbb{N}$ indexed by $S$ that restrict eventually to the identity on each ray with only finitely many exceptions where they restrict eventually to a translation. In other words, $H(S)$ is the direct limit of the Houghton groups $H(R)$ for $R \subset S$ finite. If a group $G$ acts on $S$, then it naturally acts on $\mathcal{R}(S)$ by permuting the rays, which induces an action of $G$ on $H(S)$ by automorphisms. Define the \emph{permutational Houghton group} $H(G\curvearrowright S)$ as $H(S) \rtimes G$. 

\medskip \noindent
Let $QV$ denote the group of quasi-automorphisms of the infinite binary tree $T_2$. Fixing a subset $O \subset T_2$ of size $k$, let $Q_OV$ denote the pointwise fixator of $O$ with respect to the action $QV \curvearrowright T_2$. Plausably, $H:= H(Q_OV \curvearrowright T_2)$ is a group of type $F_\infty$ with exactly $k$ focal hyperbolic structures. 

\medskip \noindent
The fact that $H$ has exactly $k$ focal hyperbolic structures should follow from the same arguments as in the proof of Theorem~\ref{thm:HoughtonPerm} combined with the fact that $Q_OV$, which is isomorphic to $QV$, satisfies Property (NL) \cite{NL}. And the fact that $H$ is of type $F_\infty$ should follow from the fact that the action $Q_OV \curvearrowright T_2$ is oligomorph with stabilisers of finite subsets of type $F_\infty$ (similarly to what happens for permutational wreath products \cite{MR3356831}). 
\end{remark}

\addcontentsline{toc}{section}{References}

\bibliographystyle{alpha}
{\footnotesize\bibliography{HypHoughton}}

\begin{thebibliography}{CCMT15}

\bibitem[ABO19]{MR3995018}
C.~Abbott, S.~Balasubramanya, and D.~Osin.
\newblock Hyperbolic structures on groups.
\newblock {\em Algebr. Geom. Topol.}, 19(4):1747--1835, 2019.

\bibitem[ABR23]{MR4585997}
C.~Abbott, S.~Balasubramanya, and A.~Rasmussen.
\newblock Higher rank confining subsets and hyperbolic actions of solvable
  groups.
\newblock {\em Adv. Math.}, 424:Paper No. 109045, 47, 2023.

\bibitem[ABR24]{MR4751864}
C.~Abbott, S.~Balasubramanya, and A.~Rasmussen.
\newblock Valuations, completions, and hyperbolic actions of metabelian groups.
\newblock {\em J. Lond. Math. Soc. (2)}, 109(6):Paper No. e12916, 72, 2024.
\newblock With an appendix by Abbot, Balasubramanya, Rasmussen and Sam Payne.

\bibitem[AR23]{MR4602410}
C.~Abbott and A.~Rasmussen.
\newblock Actions of solvable {B}aumslag-{S}olitar groups on hyperbolic metric
  spaces.
\newblock {\em Algebr. Geom. Topol.}, 23(4):1641--1692, 2023.

\bibitem[Bal20]{MR4069979}
S.~Balasubramanya.
\newblock Hyperbolic structures on wreath products.
\newblock {\em J. Group Theory}, 23(2):357--383, 2020.

\bibitem[Bal23]{SurveyS}
S.~Balasubramanya.
\newblock A survey on classifying hyperbolic actions of groups.
\newblock {\em arxiv:2310.09455}, 2023.

\bibitem[BCFS22]{BCFS}
U.~Bader, P.-E. Caprace, A.~Furman, and A.~Sisto.
\newblock Hyperbolic actions of higher-rank lattices come from rank-one
  factors.
\newblock {\em arxiv:2206.06431}, 2022.

\bibitem[BdCK15]{MR3356831}
L.~Bartholdi, Y.~de~Cornulier, and D.~Kochloukova.
\newblock Homological finiteness properties of wreath products.
\newblock {\em Q. J. Math.}, 66(2):437--457, 2015.

\bibitem[BFFG22]{NL}
S.~Balasubramanya, F.~Fournier-Facio, and A.~Genevois.
\newblock Property ({NL}) for group actions on hyperbolic spaces (with an
  appendix by a. sisto).
\newblock {\em arXiv:2212.14292, to appear in Geometry, Groups, and Dynamics},
  2022.

\bibitem[BFFZ24]{BFFZ}
S.~Balasubramanya, F.~Fournier-Facio, and M.~Zaremsky.
\newblock Hyperbolic actions of {T}hompson's group ${F}$ and generalizations.
\newblock {\em arxiv:2406.12982}, 2024.

\bibitem[Bro87a]{MR885095}
K.~Brown.
\newblock Finiteness properties of groups.
\newblock In {\em Proceedings of the {N}orthwestern conference on cohomology of
  groups ({E}vanston, {I}ll., 1985)}, volume~44, pages 45--75, 1987.

\bibitem[Bro87b]{MR914847}
K.~Brown.
\newblock Trees, valuations, and the {B}ieri-{N}eumann-{S}trebel invariant.
\newblock {\em Invent. Math.}, 90(3):479--504, 1987.

\bibitem[Cal09]{MR2527432}
D.~Calegari.
\newblock {\em scl}, volume~20 of {\em MSJ Memoirs}.
\newblock Mathematical Society of Japan, Tokyo, 2009.

\bibitem[CCMT15]{MR3420526}
P.-E. Caprace, Y.~Cornulier, N.~Monod, and R.~Tessera.
\newblock Amenable hyperbolic groups.
\newblock {\em J. Eur. Math. Soc. (JEMS)}, 17(11):2903--2947, 2015.

\bibitem[Hou79]{MR521478}
C.~Houghton.
\newblock The first cohomology of a group with permutation module coefficients.
\newblock {\em Arch. Math. (Basel)}, 31(3):254--258, 1978/79.

\bibitem[{Lee}12]{LeeHoughton}
S.~{Lee}.
\newblock {Geometry of Houghton's Groups}.
\newblock {\em arXiv:1212.0257}, 2012.

\end{thebibliography}

\Address

%

\end{document}